\documentclass{amsart}
\usepackage{verbatim,amsfonts,color,mathrsfs,stmaryrd,graphicx, mathdots}
\usepackage{xcolor}
\usepackage{tikz}
\usepackage{tikz-cd}
\usepackage{tkz-euclide}
\usetikzlibrary{patterns}   
\usetikzlibrary{positioning}
\usetikzlibrary{decorations.pathreplacing,calligraphy}

\usepackage{pgfplots}

\pgfplotsset{compat=1.6}
\pgfplotsset{soldot/.style={color=blue,only marks,mark=*}} 
\pgfplotsset{holdot/.style={color=blue,fill=white,only marks,mark=*}}

\usepackage{enumitem}
\usepackage{amssymb}

\theoremstyle{plain}
\newtheorem{Thm}{Theorem}

\newtheorem{Lem}[Thm]{Lemma}

\theoremstyle{definition}
\newtheorem{Def}[Thm]{Definition}
\newtheorem{Rmk}[Thm]{Remark}

\theoremstyle{remark}

\usepackage{hyperref}

\begin{document}


\title[Deformations of Bowen-Series maps]{Continuous deformation of the Bowen-Series map associated to a cocompact triangle group}

\author{Thomas A. Schmidt}
\address{Oregon State University\\Corvallis, OR 97331}
\email{toms@math.orst.edu}

 \author{Ay\c{s}e Y{\i}ltekin-Karata\c{s}}
\address{Bartin University\\Turkey}
\email{yiltekia@gmail.com}
\date{8 December 2023}


\begin{abstract} 
In 1979, for each signature for Fuchsian groups of the first kind, Bowen and Series constructed an explicit fundamental domain for one  group of the signature, and from this a function on $\mathbb S^1$ tightly associated with this group.  In general, their fundamental domain enjoys what has since been called both the `extension property' and the `even corners property'.   We determine the exact set of signatures for cocompact triangle groups for which this property can hold for any convex fundamental domain, and verify that for this restricted set, the Bowen-Series fundamental domain does have the property. 

To each Bowen-Series function in this corrected setting, we naturally associate four continuous deformation families of circle functions.  We show that each of these functions
is aperiodic if and only if it is surjective;
 and, is finite Markov if and only if its natural parameter is a hyperbolic fixed point of the triangle group at hand.
\end{abstract}

\maketitle

\tableofcontents

\section{Introduction}    
In 1979, for each signature for Fuchsian groups of the first kind, Bowen-Series \cite{bs}  constructed an explicit fundamental domain for one  group of the signature, and from this, a so-called boundary map, a function on $\mathbb S^1$, tightly associated with this group.    In the cocompact setting, they showed that their function is uniquely ergodic with respect to a probability measure equivalent to Lebesgue measure.     These  functions have been revisited many times, most notably in the surface group setting by, among others,  Adler-Flatto \cite{af}, Morita \cite{Morita},  Pit \cite{pit}, Los  and co-authors  \cite{losBSLmaps, losEtAlVolEntropy}, and S.~Katok and co-authors \cite{KatokUgarcoviciBdryMaps17, adamsKatokRevisit2019,adamsKatokUgarcovici2022Flex}.    Many interval maps have the flavor of boundary maps, in particular regular continued fractions and the underlying interval map are well-known to be related to the Fuchsian group $\text{PSL}_2(\mathbb Z)$.   The well-studied $\alpha$-continued fractions of Nakada \cite{Nakada} give a one-parameter deformation of the regular continued fraction map.  In an analogous manner, there is a geometrically natural way to form one-parameter deformations of the Bowen-Series maps, as \cite{KatokUgarcoviciBdryMaps17, losBSLmaps} do in the surface group setting.

One naturally expects that many of the results for the surface groups will succeed in the setting of the cocompact triangle groups.    Some aspects must be simpler, since any two cocompact triangle Fuchsian groups of the same signature are conjugate in $\text{PSL}_2(\mathbb R)$.  (In particular, the use by Bowen-Series of quasi-conformal maps, to extend their results so as to apply to all Fuchsian groups of a given signature, is unnecessary in this setting.)    We were surprised to find that a central feature of the Bowen-Series construction fails for certain cocompact triangle signatures, see Theorem~\ref{t:whichTrianglesWork}.   We restrict to the signatures where this `extension property' --- also known in the literature as the {\em even corners property}, see say \cite{SeriesGeom, BufetovSeries} --- does hold, and consider  one-parameter deformations of the Bowen-Series function for each.   Another surprise arose:  Functions so created can fail to be ergodic with respect to a probability measure equivalent to Lebesgue measure.  Indeed, one finds functions that fail to be surjective!  In Theorem~\ref{t:aperiodic} we identify exactly when our functions are surjective.

Recall that aperiodicity of a function with respect to a partition is a type of transitivity property.  As  Adler-Flatto \cite{af} prove (see their `Folklore Theorem'), when a sufficiently smooth piecewise monotone eventually expanding function is both aperiodic and Markov with respect to a finite partition, it has an ergodic probability measure equivalent to Lebesgue measure.   We show, see again Theorem~\ref{t:aperiodic},  that any surjective member of the deformation families which is Markov is also aperiodic.    The (finite) Markov property holds exactly for those parameter values which are hyperbolic fixed points of the Fuchsian group at hand, see Theorem~\ref{t:markoffIsHyper}.

\subsection{Main Results} 
 
We give a  correction to the initial sentence of [\,\cite{bs}, Section 3].    Let  ``$\mathcal F$ is a fundamental domain  for the signature" mean that there is some Fuchsian group of the given signature which has $\mathcal F$ as a fundamental domain.  To simplify notation throughout, we write $(m_1,m_2,m_3)$ to represent the standard  signature $(0; m_1, m_2, m_3)$ when speaking of  Fuchsian triangle groups. The definition of the extension property is given in Subsection~\ref{ss:theNet}.%

 \begin{Thm}\label{t:whichTrianglesWork}   Suppose that $(m_1, m_2, m_3)$ is the signature of a 
cocompact hyperbolic Fuchsian triangle group.  If more than one $m_i$ is odd, then  no convex fundamental domain  for the signature  has the extension property.
Otherwise,  the Bowen-Series fundamental domain for this signature does have the extension property. 
\end{Thm}

 We define $\mathscr E$ be the set of those signatures $(m_1, m_2, m_3)$ corresponding to  a cocompact hyperbolic Fuchsian triangle group such that the Bowen-Series fundamental domain is a non-degenerate quadrilateral with the extension property,  and insist on a particular ordering of the $m_i$; see Definition~\ref{d:WithExtensionProperty}.   The Bowen-Series function is defined in terms of the side pairing elements of the Bowen-Series fundamental domain,    $T_i, 1\le i \le 4$; see Subsection~\ref{ss:bsFun}.   For $(m_1, m_2, m_3) \in \mathscr E$,  the function $f$ is eventually expanding since each $T_i$ is applied on a subset of its isometric circle.
Set $(n_1, n_2, n_3, n_4) = (m_3, m_2/2, m_3, m_1/2)$.   These latter values control aspects related to the partition $\mathcal P$ of $\mathbb S^1$ associated to $f$, see Subsection~\ref{ss:theFun}.    

Recall that a function $g$ is called {\em aperiodic} with respect to a partition of its domain of definition if there is a finite compositional power of the function that maps the closure of each partition element onto the whole domain.   Note that we use the term {\em Markov} following the convention that it means what is sometimes called {\em finite Markov}.   The main property this entails is:   There is a finite partition $\{I_1, \dots, I_k\}$ of the domain such that if the image of a partition element meets any partition element, then the image contains that partition element.  That is,  $g(I_i) \cap I_j \neq \emptyset$ implies  $g(I_i) \supset I_j$.    Bowen-Series showed that their eventually expansive function $f$ is both aperiodic and Markov, and therefore has an invariant probability measure equivalent with Lebesgue measure.   

 There are four corresponding {\em overlap intervals} $\mathscr O_i$ where appropriately replacing $f(x) = T_i x$ by $x \mapsto T_{i-1}x$  for all $x$ in some subinterval determined by a parameter $\alpha$ leads to an eventually expansive function $f_{\alpha}$, see Subsection~\ref{ss:defFunPart}.

 \begin{Thm}\label{t:aperiodic}    Fix $(m_1,m_2,m_3) \in \mathscr E$ and let $f$ denote its Bowen-Series function.    Then for $\alpha \in \mathscr O_i$ the function $f_{\alpha}$ of \eqref{e:ThatsOurFunction} is surjective   if and only  if any of the following conditions holds:  
 \begin{itemize}
\item[(i)] $n_i>2$,
\item[(ii)] $n_i=2$ and $n_{i+2}>2$, 
\item[(iii)]   $\alpha$ belongs to the closure of the set of points $x \in \mathscr O_i$ such that $f^{n_i}(x) = f^{n_i-1}(T_{i-1} x)$
 \end{itemize}
 
 Furthermore if $f_{\alpha}$ is Markov with respect to  the partition $\mathcal P_{\alpha}$ which is defined in Subsection~\ref{ss:defFunPart},   then  $f_{\alpha}$ is aperiodic with respect to $\mathcal P_{\alpha}$ if and only if  $f_{\alpha}$ is surjective.
\end{Thm}

 \begin{Thm}\label{t:markoffIsHyper}   Let $\Gamma$ be the group from which $f$ is constructed.   The map $f_{\alpha}$ is Markov with respect to $\mathcal P_{\alpha}$ if and only if $\alpha \in \mathscr O$ is a hyperbolic fixed point of  $\Gamma$. 
\end{Thm}
 
 Further results about these maps are given in the Ph.D. dissertation of the second named author, \cite{aykDiss}.

\subsection{Thanks}  We thank the referee for suggestions improving the exposition, and for certain additions to the bibliography.

\section{Background} 
\begin{figure}[h]
\scalebox{.3}{
\includegraphics{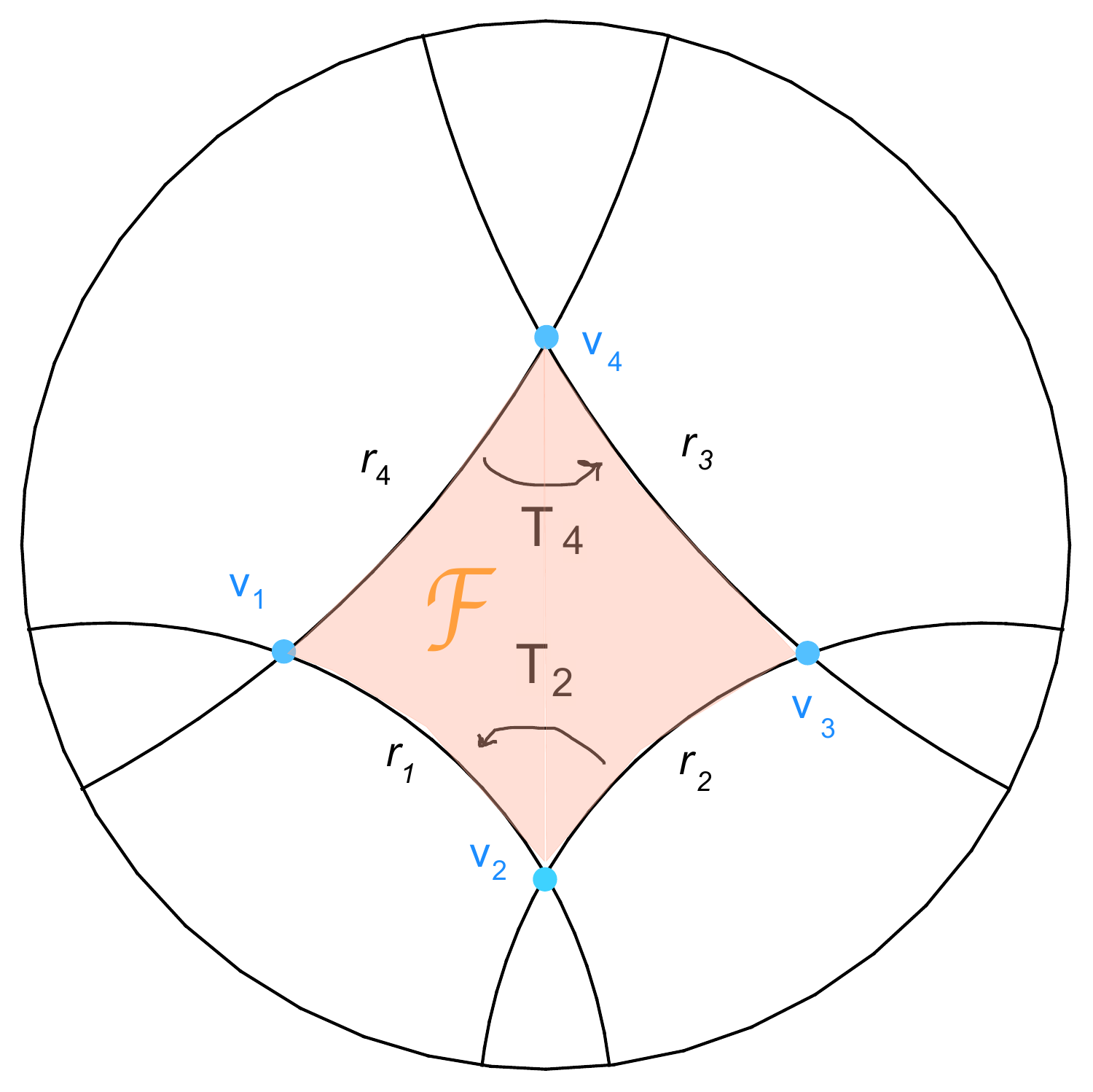} 
}
\caption{The Bowen-Series fundamental domain $\mathscr F$, here $(m_1, m_2, m_3) = (6,6,3)$.}%
\label{f:basicFunDom}%
\end{figure}
  
\subsection{The Bowen-Series fundamental domain}\label{ss:bsFun} Bowen-Series \cite{bs} construct one fundamental domain, within the disk model $\mathbb D$ of the hyperbolic plane, per each signature of Fuchsian groups.   As already stated, in the setting of triangle groups, where since the genus is zero, we use a simplified signature of $(m_1, m_2, m_3)$ and will simply call this the signature.  Recall that permutations of the indices result in signatures of isomorphic groups.

  Given $(m_1, m_2, m_3)$, their construction results in a fundamental domain  $\mathscr F$ which is a  quadrilateral; note that  we follow the convention that a `side' cannot be paired with itself, as in \cite{beardon} but not  as in \cite{bs} or \cite{pit}.
We normalize and label, see Figure~\ref{f:basicFunDom},  so that $\mathscr F$ has vertices $v_1, v_3$ each of internal angle $\pi/m_3$ and $v_2, v_4$ of respective internal angles $2 \pi/m_2, 2 \pi/m_1$ with $v_2, v_4$ lying on the central vertical line and where the $v_i$ are in counter clockwise order.     
  
  In all that follows,   indices $i\pm 1$ are to be the appropriate representative of  $i\pm 1 \pmod 4$.   Redundantly labeling the sides as $r_i, \ell_i,  1 \le i \le 4$ where $r_i$ is the side meeting $v_i$ at its right (facing inward) and $\ell_i$ at its left, we denote by $T_i$ the side pairing taking $r_i$ to its paired side.   Thus, letting 
 \[\sigma(i):=\begin{cases}
i-1\;    \text{     if }i\text{ is even}; \\
i+1\; \text{     if }i\text{ is odd}, 
\end{cases}
\]
we have $T_i(r_i)=r_j$ whenever $\sigma(i)=j$. Similarly, we define 
\begin{equation}\label{e:rho}
\rho(i) :=\sigma(i)+1
\end{equation}
and find that $\rho(i)=j$ implies that $T_i(v_i)=v_j$.   Accordingly, for all $i$ we have

\begin{equation}\label{relations}
\begin{split}
T_{\sigma(i)}T_i&=I\\
T_{\rho(i)}T_{i-1}&=I\\
T_{i-1}(v_i)&=v_{\rho(i)} \text{    and     }T_i(v_i)=v_{\rho(i)}\,.
\end{split}
\end{equation}

   In particular,  $T_2(v_2) = v_2, \, T_4(v_4) = v_4$ while (up to orientation) $T_1(r_1) = r_2, \,T_3(r_3) = r_4$ and $T_2 = T_{1}^{-1}, \,T_4 = T_{3}^{-1}$.      We have that $\{v_1, v_3\}$ forms a cycle, as   $T_1(v_1) = v_3, \, T_3(v_3) = v_1$.    We find that $T_2 T_4$ fixes $v_1$ and rotates through an angle of  $2\pi/m_3$.     
  
  Key to their use of $\mathscr F$, the side $r_i$ is contained in the isometric circle of $T_i$ for each $i$. 
  
 Note also that $\rho^2$ is the identity.   The relations \eqref{relations} hence imply 
\begin{equation}\label{e:runAround} 
T_{\rho^{2n_i-1}(i)}T_{\rho^{2n_i-2}(i)}\cdots T_{\rho^2(i)}T_{\rho(i)}T_i= (T_{\rho(i)}T_i)^{n_i} = I.
\end{equation}

  \subsection{The net of a fundamental domain.}\label{ss:theNet}      Suppose that $\mathcal F$ is a convex fundamental domain for some Fuchsian group $\Gamma$.  Let $S$ be the set of the sides of $\mathcal F$ and for each $s\in S$ let $g(s)$ be the geodesic containing $s$.  Following Adler-Flatto \cite{af}, we say that $\mathcal F$ has the  {\em extension property} if 
  \[  g(s) \cap \bigcup_{T\in \Gamma}\, T(\, \mathring{\mathcal F}\,) = \emptyset\;\; \forall s\in S.\]
 
Following \cite{bs}, the {\em net} $N$ of $\mathcal F$ is the set of   images under $\Gamma$ of the sides of $\mathcal F$.   That is, 
  \[ N = \bigcup_{T\in \Gamma, s \in S}\, T(\,  s\,).\]

    The union of the sides of $\mathcal F$ is its boundary hence $N$ is the union of the images of the boundary of $\mathcal F$.  Since the images of $\mathcal F$ tesselate,   it follows that any geodesic which avoids the images of the interior of $\mathcal F$ must be contained in $N$, and vice versa.  More formally, we have the following, showing that the second part of Condition $(*)$ in  \cite{bs} is equivalent to the extension property.
  
 \begin{Lem}\label{l:netIsNet}   Suppose that $\mathcal F$ is a convex fundamental domain for a Fuchsian group $\Gamma$.  Then the extension property holds for $\mathcal F$ if and only if 
  \[  g(s) \subset N \;\; \forall s\in S.\] 

\end{Lem}

The following is an observation of Adler-Flatto \cite{af}; see the comment below their Theorem~1.1, p.~239.
 \begin{Lem}\label{l:projToGeos}   Suppose that $\mathcal F$ is a convex fundamental domain for a Fuchsian group $\Gamma$.   If the boundary of $\mathcal F$  projects to closed geodesics, then 
 the extension property holds for $\mathcal F$.
\end{Lem}
\begin{proof}  Suppose first that some side $s$ projects to a closed geodesic.  Then  $g(s)$ projects to (wrap around) this same geodesic.  Equivalently, $g(s)$ is contained in the $\Gamma$-orbit of $s$. That is, $g(s) \subset N$.    

Similarly, if there is a sequence of sides $s_1, \dots, s_n$ which projects to a closed geodesic, then each of the $g(s_i)$ projects to this geodesic.   Thus, each $g(s_i)$ is contained in $N$.
\end{proof}  
  
\begin{figure}[h]
\scalebox{0.3}{
\includegraphics{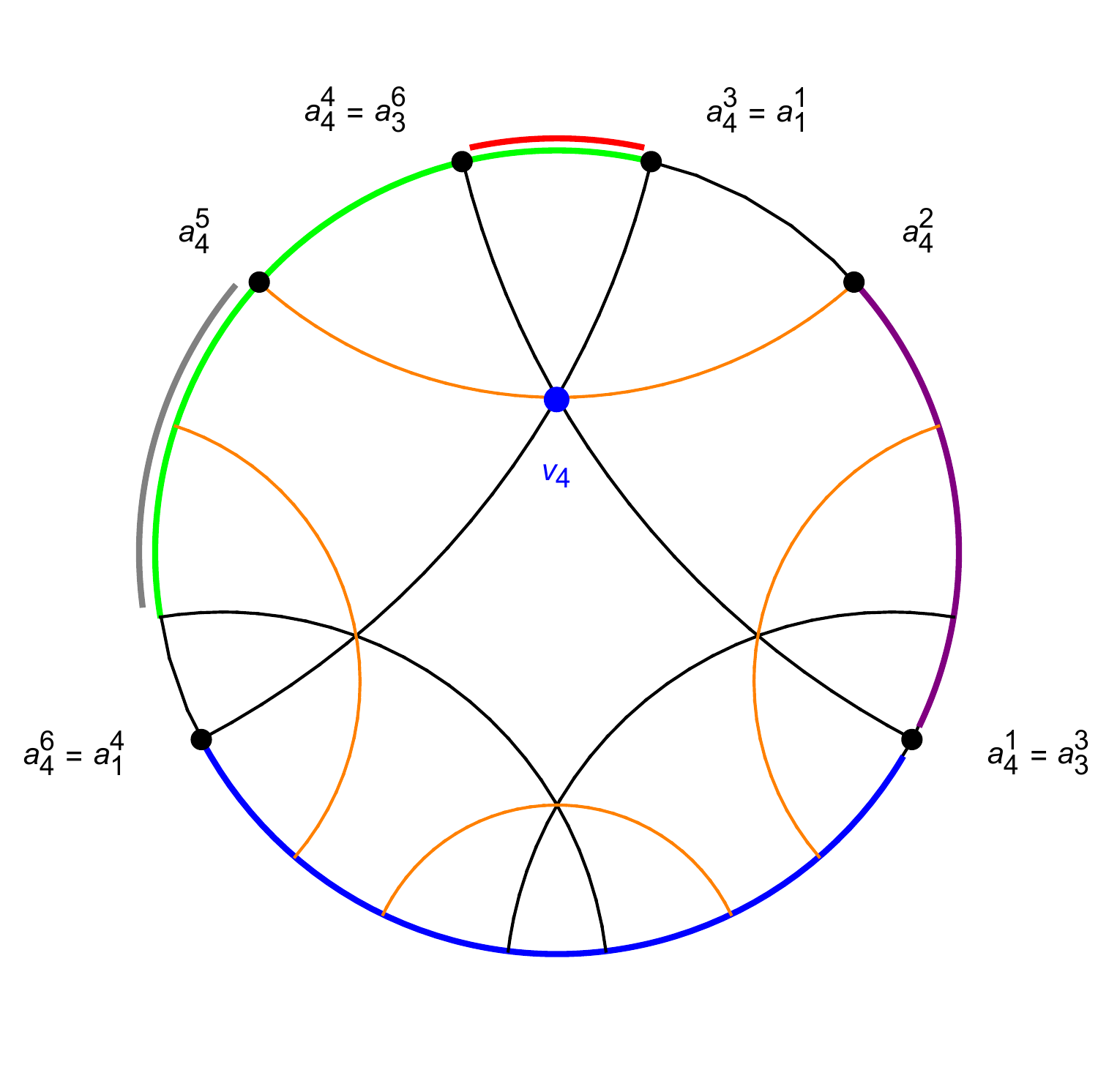}
}
\caption{The Bowen-Series fundamental domain for   $(m_1, m_2, m_3) = (6,6,3)$ has the extension property.   Note that here,  $n_i = 3$ for all $i$.  The Bowen-Series map, see \eqref{e:defF},  $f$ is such that $f(x) = T_4(x)$ on $[a_{4}^{n_4}, a_{1}^{n_1})$, in green.  Expansivity is preserved upon replacing the action of $T_4$ by $T_3$ in the overlap interval $\mathscr O_4 = L_{3}(v_4)$, hinted by red; see \eqref{e:ohEye}.   Left and right intervals in accordance with Definition~\ref{LRpartition}: In blue,  $R_1(v_4)$; in purple, $R_2(v_4)$; continuing counter clockwise, one finds $R_3(v_4)$ and then the $L_j(v_4)$ in increasing order of $j$. Hinted in gray, $\mathcal A_4$, see \eqref{e:defAsubI}.}%
\label{f:netEnds}%
\end{figure}

Following \cite{pit}, for each vertex $v \in \mathcal F$ we let $N_v$ be the set of geodesics passing through $v$ that contain the image of some side of $\mathcal F$.  That is, 
\[ N_v = \{g \; \text{geodesic} \;\vert\; v \in g;\;   \exists T \in \Gamma,  \exists s\in S,  Ts \in g\}.\]

 Suppose now that  $\mathcal F$ has the extension property.   Then Lemma~\ref{l:netIsNet}  implies that   $N_v \subset N$.   If furthermore $v$ is not an ideal point (that is, $v\notin \partial \mathbb D$), then $N_v$  is  a finite set: The tesselation by images of $\mathcal F$ meets $v$ in  finitely of these images, none of whose interiors the $g\in N_v$ can meet.\\ 

The following is a variant of a step in the proof of  [\,\cite{bs}, Lemma 2.3], see also [\,\cite{pit}, Proposition~2.2].
 \begin{Lem}\label{l:crossingFailsExtenProp}   Suppose that $\mathcal F$ is a convex fundamental domain for a Fuchsian group $\Gamma$ of finite covolume. Assume that $\mathcal F$  is not a degenerate quadrilateral.    If $\mathcal F$ satisfies the extension property, then for any $g, g' \in N$, 
 we have that $g\cap g'$  is either:  empty,   a vertex of $\mathcal F$, or equal to $g(s)$ for some side $s$.
\end{Lem}

\begin{proof}  The other cases being trivially verified, we take $g, g'$ to be distinct and even such that there are  distinct vertices $v, v'$ of $\mathcal F$ with $g \in N_v, g'\in N_{v'}$.     We argue by contradiction.   Suppose $v, v', g, g'$ are such that $g \cap g' = P$ with the point $P$ not lying on a side of $\mathcal F$.    Let  $\Delta := vPv'$ be the closed triangle with these three points as vertices.   Since neither of $g, g'$ can meet the interior of $\mathcal F$, convexity shows that  $\Delta$ contains a consecutive vertex $v''$ of one of $v, v'$.  Relabel such that  $v, v''$ are consecutive vertices.  If $v'' \neq v'$, then let $v''' \neq v$ also be consecutive to $v''$ and $s''$ their shared side.  Convexity also shows that $g(s'')$ meets $g$.     That is, by replacing $v'$ by $v''$ and $g'$ by $g(s'')$ as necessary,  we may and do assume that $v, v'$ are consecutive vertices.

 Now let $U \in \Gamma$ be the side pairing element such that $U^{-1} (s)$ is also a side.  
  Then  $U(\mathcal F) \cap \Delta \neq \emptyset$.    Since the sides of $\Delta$ cannot intersect the interior of    $U(\mathcal F)$, we have that  $U(\mathcal F) \subset \Delta$.     By convexity, the sides $t, t'$ of $U(\mathcal F)$ meeting $s$ at $v, v'$ respectively must be such that $g(t') \cap g(t'')$ meet in a point  of $\Delta$.  But,  by [\,\cite{bs}, Lemma~2.2] this is impossible (under their running assumption of the extension property).
\end{proof} 

 \begin{Rmk}\label{r:BSrunningUpSide}  The proof of  \cite{bs} for their Lemma~2.2 does not explicitly mention a special case that could be of interest in the setting of cocompact triangle groups.    Given some triangle $vPv'$ as in the proof above,   
  they implicitly find a sequence $(r_k)_{k\ge 1}$, with $r_k$ the intersection  of  $g(t_k), g(t'_k)$ for distinct sides $t_k, t'_k$ (separated by exactly one side, say $s_k$) of  some $U_k(\mathcal F)$ such that for $k\ge 2$, $r_k$ is in the triangle of endpoints $P_{k-1}$ and the endpoints of $t_{k-1}, t'_{k-1}$  meeting $s_{k-1}$.   See [\,\cite{bs}, Figure~1].    A contradiction is raised, since the area of the original triangle is finite.   
 
 The case that they do not explicitly treat is where the sequence of the $r_k$ is eventually constant.     In that case, we can simply assume that the sequence is constant.   We can thus suppose that $g(t_1),g(t'_1)$ and $g(t_2),g(t'_2)$  all meet at $P_1$.   As in that proof, one also has that each $U_k(\mathcal F)$ is the image of $U_{k-1}(\mathcal F)$ by way of a side pairing to $s_{k-1}$, and that $s_k \neq s_{k-1}$.  Hence,     $U_2(\mathcal F)$ is a quadrilateral and thus so is each $U_k(\mathcal F)$.    Therefore, the  complement of $U_1(\mathcal F)$ in its triangle is again a triangle.   Similarly, each  union $\cup_{i=1}^k\, U_i(\mathcal F)$ gives a complement in this first triangle that is again a triangle.  Thus, the initial triangle cannot be the union of finitely many of these quadrilaterals, but  again by area considerations it cannot contain infinitely many of them; a contradiction is reached in this case as well.     
\end{Rmk}
 
\subsection{The Bowen-Series function}\label{ss:theFun}   Fix a cocompact Fuchsian triangle group signature $(m_1, m_2, m_3)$.       Assume that the  Bowen-Series fundamental domain $\mathscr F$ has the extension property.    Theorem~\ref{t:whichTrianglesWork}, below, shows that this is a restrictive assumption.

\subsubsection{Definition of function $f$ and its partition $\mathcal P$}\label{ss:defBasicProps}  To simplify notation, let  $N_i = N_{v_i}$; let also $n_i=\#N_i$.   Since $T_4$  fixes $v_4$ and rotates through an angle of $2 \pi/m_1$,  it is always the case that $2 n_4 \ge m_1$ and certainly our assumption of the extension property then shows that equality must hold: $2 n_4 = m_1$;  similarly,  $2 n_2 = m_2$ and $n_1= n_3 = m_3$.

  We label the endpoints of the $g \in N_i$, see Figure~\ref{f:netEnds}.   For each $i$,   let  $a_i^1$ be the endpoint of $g(\ell_i)$ which is on the same side of $v_i$ as is $v_{i-1}$.   Then continue labeling the endpoints of the $g \in N_i$ in counter clockwise order as $a_{i}^{2}, \cdots, a_{i}^{2n_i}$.      Following \cite{bs}, let $W_i = \{a_i^1, \cdots, a_i^{2n_i}\,\}$.

 Of course $a_{i}^{n_i+1}$ is the other endpoint of $g(\ell_i)$.    Since $g(\ell_{i+1}) \cap g(\ell_i) = v_i$, by Lemma~\ref{l:crossingFailsExtenProp} no other $g \in N$ can have an endpoint between  $a^{1}_{i+1}$ and $a_{i}^{n_i+1}$.  Since $g(\ell_{i+1}) \in N_i$,   one has the key identities
$a_{i}^{n_i} = a^{1}_{i+1}$  and $a_{i}^{2 n_i} = a^{n_{i+1}+1}_{i+1}$.  We find that the intervals $[a_{i}^{1}, a_{i}^{n_{i}}\,), \; 1\le i\le 4$ partition $\mathbb S^1$.\\

 In our notation, the function defined in \cite{bs} is 
\begin{equation}\label{e:defF}
\begin{aligned}
f: \mathbb S^1 &\to \mathbb S^1\\
     x   &\mapsto T_{i-1}(x)\;\;\;\;\text{if}\; x \in [a_{i}^{1}, a_{i}^{n_{i}}\,).
\end{aligned}
\end{equation}

\begin{figure}[h]
\scalebox{1.1}{
\includegraphics{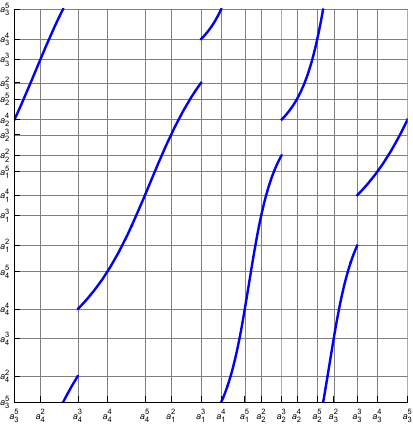} 
}
\caption{Plot representing the Bowen-Series function $f$ for signature $(6,6,3)$.  The partition $\mathcal P$, consisting of the 16 intervals indicated along the axes, is a Markov partition for $f$.  The true plot lives on a torus, here the torus is cut, using `argument' functions.  Note that, on the torus,  there are four connected components of the graph of $f$.}  
\label{f:plotBS}%
\end{figure}

The map $f$ is (eventually) expansive: For each $i$, the side $r_{i-1}$ lies on the isometric circle of $T_{i-1}$ and thus $T_{i-1}$ increases Euclidean distances on the  arc subtended by the geodesic containing $g(r_{i-1}) = g(\ell_{i})$.  This arc is   $(a_{i}^{1}, a_{i}^{n_{i}+1}\,)$.   Since $[a_{i}^{1}, a_{i}^{n_{i}}\,)$ is a subset of the arc, the eventual expansivity of $f$ easily follows.   Let $W$ be the union of the $W_i$\,; Bowen-Series show that $f$ is a Markov map with respect to the {\em partition}  $\mathcal P$  of $\mathbb S^1$ whose endpoints form $W$.  In particular, $f(W) = W$.

\subsubsection{ More partitions} 

The following partitions of $\mathbb S^1$  are also from \cite{bs}.  
\begin{Def}\label{LRpartition}
Let 
\[L_{j}(v_i)=[a_i^{2n_i-j}, a_i^{2n_i-j+1}), \;R_j(v_i)=[a_i^{j-1}, a_i^{j}),\quad\text{ where } 1\leq j\leq n_i,\]
and we take $j-1=2n_i$ when $j=1$.  
\end{Def}
 
 It follows from \eqref{wpts},  that each of  the $L_r(v_i), 2\leq r\leq n_i$ and  $R_s(v_i), 3\leq s\leq n_i$ is itself an element of $\mathcal P$.  As well, each of 
$L_1(v_i),  R_1(v_i), R_2(v_i)$ is a union of  elements of $\mathcal P$.
Furthermore, for each $i$,  the `left' and `right' intervals partition $\mathbb S^1$,  
\begin{equation}\label{e:leftRightPartition}
\mathbb S^1 = \bigsqcup_{j=1}^{n_i} \big(\,L_j(v_i)\sqcup R_j(v_i)\,\big)\,.
\end{equation}

   We use further notation from \cite{bs}: For each $i$,  let 
\begin{equation}\label{e:defAsubI} \mathcal A_i=L_1(v_i)\setminus L_{n_{i+1}}(v_{i+1}).
\end{equation}
  Thus,  $f(x) = T_i x$ holds exactly on $ \mathcal A_i \cup \bigcup_{k=2}^{n_i}\, L_k(v_i)$.

By the conformality of each $T_i$, we have for all $i$ and  $1\leq k\leq 2n_i$ 
\begin{equation}\label{wpts}
T_{i-1}(a_i^k)=a_{\rho(i)}^{k-1}\quad\text{and}\quad T_i(a_i^{k})=a_{\rho(i)}^{k+1}\;,
\end{equation} 
with values taken appropriately modulo $2n_i$.

Using \eqref{wpts} and the conformality of the transformations in $\Gamma$, for each $i$,

\begin{equation}\label{e:partition}
\begin{split}
f\mid_{L_{r}(v_i)}&=T_i,\quad \quad f(L_r(v_i))=L_{r-1}(v_{\rho(i)})\,\,\quad\text{for   }2\leq r\leq n_i\,;\\
f\mid_{R_{s}(v_i)}&=T_{i-1},\quad f(R_s(v_i))=R_{s-1}(v_{\rho(i)})\,\,\quad\text{for   }2\leq s\leq n_i\,.
\end{split}
\end{equation}

\subsection{Matching $f$-orbits of $x$ and $ T_{i-1}(x)$ for $x \in \mathscr O$.}\label{ss:matching}     In this subsection, we  continue to establish notation, and make more precise a result of \cite{bs}.  

At the expense of more doubling of notation,  
for each $i, 1\le i \le 4$,  define the corresponding {\em overlap} interval 
\begin{equation}\label{e:ohEye}
\mathscr O_i = [a_{i}^{n_i}, a_{i}^{n_i+1}).
\end{equation}
  That is,  $\mathscr O_i = L_{n_i}(v_i)$, and hence $\mathcal A_i = L_1(v_i)\setminus \mathscr O_{i+1}$.   Let $\mathscr O = \cup_{i=1}^4\, \mathscr O_i$\,.

 \begin{Lem}\label{l:keyContainers}  Each $\mathscr O_i$  is   a leftmost subinterval of each of $R_2(v_{i+1})$, a rightmost subinterval of $L_1(v_{i-1})$ and is contained in $R_1(v_{i+2})$.  No other $R_j(v_k)$ or $L_j(v_k)$ meets  $\mathscr O_i$. 
 \end{Lem}
\begin{proof}  This is verified directly from  the definition of the intervals. 
\end{proof}

With our usual conventions, let  
\[\theta(i)=\rho^{n_i-1}(i),\]  
where $\rho$ is given in \eqref{e:rho}.   
 With $i$ fixed,  let 
\begin{equation}\label{e:indexSequence} 
i_0 = i \text{ and  } i_k = \theta(i_{k-1}) +1, \; \forall k \in \mathbb N.  
\end{equation}

  Given $x \in \mathscr O_i$, let  
\begin{equation}\label{e:orbitSequences} 
 \begin{cases}  x_0 = x, &\text{and }\; x_{k+1} = f^{n_{i_k}-1}(x_{k})\; \forall k \ge 0;\\
                         y_0 = T_{i-1} x, &\text{and }\; y_{k+1} = f^{n_{i_{k}}-1}(y_{k})\; \forall k \ge 0.\\
 \end{cases} 
\end{equation}
Thus,  for each $\ell \in \mathbb N$ letting 
\[r_{\ell} =  \sum_{k=0}^{\ell-1}\, (n_{i_k}-1),\]
for $k\ge 2$,  we have $x_k = f^{r_k}(x) = f^{r_k - r_{k-1}}(x_{k-1}) $ and similarly for $y_k$.    We think of $x_k \mapsto x_{k+1}$ as a {\em giant step}  in the orbit of $x$ and each  $f^j(x_k)\mapsto f^{j+1}(x_k)$ with $j \le n_{i_k}-1$ as a {\em baby step},  and similarly for maps on the $f$-orbit of of $y_0$.

 Still with $i$ fixed,  for each $\ell \in \mathbb N$, let  
 \[ M_{\ell} = M_{\ell}^{(i)} = \{x \in \mathscr O_i\,\vert\, f(x_{\ell}) = y_{\ell},\;\text{with minimal }\ell \}.\]
Thus, for $x \in M_{\ell}$ setting $p=r_{\ell}$ we have that 
\[ f^{p+1}(x) = f^{p}(\, T_{i-1}(x)\,).\]  
In Lemma~\ref{l:allMatching}, below,  we show that the $M_{\ell}$  partition $\mathscr O_i$.

 \bigskip
 
   Fix $i$.  For $x \in \mathscr O_i$ by   \eqref{e:partition} and the definition of the Bowen-Series function $f$,  
\begin{equation}\label{e:landingPattern} 
f^{n_i}(x) = \begin{cases} T_{\theta(i)}f^{n_i-1}(x)&\text{if}\;f^{n_i-1}(x)\in  \mathcal A_{\theta(i)} = L_1(v_{\theta(i)})\setminus \mathscr O_{\theta(i)+1};\\
\\
                                         T_{\theta(i)+1}f^{n_i-1}(x)&\text{if}\;f^{n_i-1}(x)\in \mathscr O_{\theta(i)+1}.
                   \end{cases}
\end{equation}
The first statement of the following shows that for every $x \in  \mathscr O_i$ one of the two conditions above does indeed hold.

This lemma and its proof are extracted from the proof of [\,\cite{bs}, Lemma~2.4]. 
 \begin{Lem}\label{l:classificationOKandOrbit1} The map $f^{n_i-1}$ restricted to  $\mathscr O_i$  gives a homeomorphism onto $L_1(v_{\theta(i)})$.    Furthermore, for all $x\in  L_{n_i}(v_i)$,
\begin{equation}\label{orbit1}
T_{\theta(i)}f^{n_i-1}(x)=f^{n_i-1}(\, T_{i-1}(x)\,).
\end{equation}
Moreover,  $f^{n_i}(x) = f^{n_i-1}(\, T_{i-1}(x)\,)$ holds on  the inverse image of $\mathcal A_{\theta(i)}$ under the homeomorphism. 
\end{Lem}

\begin{proof}  [Sketch]
The transitivity of the equalities in the top line of \eqref{e:partition} gives  
\begin{equation}\label{e:babystepsLeftOrbit}
\begin{aligned}
f^{n_i-1}:  \mathscr O_i = L_{n_i}(v_i) &\xrightarrow{\sim} L_1(v_{\rho^{n_i-1}(i)})\\
                                x &\mapsto T_{\rho^{n_i-2}(i)}\cdots T_{\rho(i)}T_i(x).
\end{aligned}                                
\end{equation}  
By definition,  $\theta(i) = \rho^{n_i-1}(i)$ and thus the first statement holds.   Due to \eqref{e:landingPattern}, the first and second statements imply the third statement.
\medskip

To show that \eqref {orbit1} holds, first note that 
\begin{equation}\label{e:ThetaValues}
 \theta(i) = \begin{cases}  i&\text{ if } 2 \mid i;\\
                                           i &\text{ if }  2 \nmid i, \, 2 \nmid n_i\,;\\
                                           \rho(i)&\text{ if }  2 \nmid i, \, 2 \mid n_i\,.
                    \end{cases}
\end{equation}

From \eqref {e:babystepsLeftOrbit},\eqref{e:ThetaValues} and the fact that $\rho^2$ is the identity, we find
\[
f^{n_i-1}(x)=   \begin{cases}  T_i^{n_i-1}(x)\in L_1(v_i)   &\text{if}\;\; 2\mid i\,;\\
  \\
                                               (T_{\rho(i)}T_i)^{\frac{n_i-1}{2}}(x)\in L_1(v_i)&\text{ if }  2 \nmid i, \, 2 \nmid n_i;\\
\\
                                               T_i(T_{\rho(i)}T_i)^{\frac{n_i-2}{2}}(x)\in L_1(v_{\rho(i)})&\text{if}\;\;   2\nmid i,\, 2\mid n_i\,.\\
                                             
                       \end{cases}
\]

 On the other hand, using \eqref{wpts}  one finds $T_{i-1}(x)\in R_{n_i}(v_{\rho(i)})$ and, similarly to the above \eqref{e:partition} gives 
 \begin{equation}\label{e:rightLandsInR} 
 f^{n_i-1}(T_{i-1}(x))=T_{\rho^{n_i-1}(i)-1}\cdots T_{\rho^2(i)-1}T_{\rho(i)-1}T_{i-1}(x)\in R_1(v_{\rho^{n_i}(i)}).
 \end{equation}      
Again by the fact that $\rho^2$ is the identity, one has 
\[
f^{n_i-1}(T_{i-1}(x))=   \begin{cases}  T_{i-1}^{n_i}(x)\in R_1(v_i) &\text{if}\;\; 2\mid i\,;\\
\\
T_{i-1}(T_{\rho(i)-1}T_{i-1})^{\frac{n_i-1}{2}}(x)\in R_1(v_{\rho(i)})&\text{ if } 2 \nmid i, \, 2 \nmid n_i;\\
\\
                                               (T_{\rho(i)-1}T_{i-1})^{\frac{n_i}{2}}(x)\in R_1(v_i)&\text{if}\;\;   2\nmid i,\,2\mid n_i\,;
                                    \end{cases}
\]

If $i$ is even, then   \eqref{relations} and \eqref{e:runAround} give $T_{i}^{-1}=T_{i-1}$ and  $T_{i}^{2n_i}=I$. We find 
\[ f^{n_i-1}(T_{i-1}(x))=T_{i-1}^{n_i}(x)=T_i^{-n_i}(x)=T_i^{n_i}(x)=T_iT_i^{n_i-1}(x)=T_{\theta(i)}f^{n_i-1}(x).\] 
The other cases offer only slightly more difficulty.
\end{proof}

 \begin{Lem}\label{l:allMatching}   Fix $i \in \{1, \dots, 4\}$. With notation as above, 
\[\mathscr O_i = \cup_{\ell=1}^{\infty}\, M_{\ell}\,.\]
Furthermore, for $x \in M_{\ell}$  and $k <\ell$, we have   $x_k \in \mathscr O_{i_k}$. Moreover, the $f$-orbit segment from $x$ to $f^{r_{\ell}-1}(x)$ meets $\mathscr O$ in exactly the $x_k, k<\ell$.
\end{Lem}

\begin{proof}  Set $i_0 = i$. Lemma~\ref{l:classificationOKandOrbit1}  shows that   $M_1$ is the inverse image of $\mathcal A_{i_1 -1}$ under the restriction of $f^{r_1}$  to $\mathscr O_i$.
   This is 
a non-empty subinterval of $\mathscr O_i$ whose complement is an interval.   Applying the lemma again shows that on this complement  $f^{r_2}$ is the composition of homeomorphisms and hence is a homeomorphism onto its image.  
Hence, $M_1 \cup M_2$ is a non-empty subinterval  whose complement is  an interval  to which  $f^{r_3}$ restricts to be a homeomorphism.   We can clearly repeat this argument indefinitely:  Each $M_{\ell}$  is non-empty,  and their  union over all $\ell$ gives a subinterval of  $\mathscr O_i$ whose complement is an interval.  

We now have that $\mathscr O_i \setminus\cup_{\ell=1}^{\infty}\, M_{\ell}$ is either empty or a finite length interval, say $J$.
Since
$f$ is eventually expansive (and $f$ itself is never contractive),  the length of  $f^{r_{\ell}}(J)$ is unbounded as $\ell\to \infty$. 
However,   $f^{r_{\ell}}(J)\subset L_{n_{i_{\ell}}}(v_{i_{\ell}})$ for every $\ell$.   This contradiction shows that $\mathscr O_i \setminus\cup_{\ell=1}^{\infty}\, M_{\ell}$ is in fact empty.

 By definition, $x_0 \in \mathscr O_{i_0}$.   For $1\le k <\ell$,  Lemma~\ref{l:classificationOKandOrbit1} shows that  $x_k \in \mathscr O_{\theta(i_{k-1})+1}$.  Since $\theta(i_{k-1})+1 = i_k$, we have that $x_k$ does belong to $\mathscr O_{i_k}$.
 
 Finally, by \eqref{e:partition} each baby step is in an interval of the form $L_j(v_{\ell})$ with $j \neq 1$.  
Thus, by Lemma~\ref{l:keyContainers} none of these is in $\mathscr O$.
\end{proof}

\subsection{Strong orbit-equivalence }\label{ss:Strong}   Orbit equivalence between a function $g$ and a group $G$ acting on $\mathbb S^1$ holds when for any pair  $(x,y) \in \mathbb S^1 \times \mathbb S^1,  \exists T \in G$ such that $y = Tx$ if and only if $g^p(x) = g^q(y)$ for some $p,q \ge 0$.   Bowen and Series show the orbit equivalence of their function $f$ and the  group action of $\Gamma$, at least up to a finite number of exceptions  in $\mathbb S^1 \times \mathbb S^1$.   Morita \cite{Morita}  shows that the exceptional set is empty.

We use notation introduced in \cite{pit}.  With $g, G$ as above, for  each $x$, let  $\gamma_1[x] = \gamma[x] \in G$ such that $g(x) = \gamma[x] x$, and for $p\in \mathbb N$  let $\gamma_p[x] = \gamma[\, \gamma_{p-1}[x] x\,]$.   Thus, for any $p\in \mathbb N$, one has   $g^p(x)   = \gamma_p[x]\, x$.     
A property observed by  Morita  to hold for Bowen-Series functions,   
was named by Pit \cite{pit}  {\em strong orbit-equivalence}:    $\forall x \in \mathbb S^1, \forall T \in G, \; \exists p,q \ge 0$ such that  $\gamma_p[x] = \gamma_q[\, T x]\, T$ in $G$.\\

We now mildly generalize some results in \cite{pit}.    
 \begin{Lem}\label{l:strongExpanGivesPerIffHyper}    Suppose that a function $g$ is expansive, given on intervals by elements of a Fuchsian group $G$ acting on $\mathbb S^1$.   Then every point of finite $g$-orbit is a  fixed point of $G$.
 \end{Lem}
\begin{proof}  Suppose that $x \in \mathbb S^1$ is periodic, with say $g^p(x) = x$.   Then  $\gamma_p[x] \in G$ fixes $x$.   Since $g$ is given piecewise by elements of $G$, there is a non-empty interval $I$ containing $x$ such that $g^p$ agrees with the action of $\gamma_p[x]$ on $I$.   Since $g$ is expansive, $g^p$ here is not the identity.  We conclude that $x$ is a fixed point for $G$.    If $y$ is any point of finite $g$-orbit, then $y$  must be have a periodic point $x$ in its forward orbit.   Say $g^q(y) = x$.   We then find that $(\gamma_q[y])^{-1}\cdot \gamma_p[x]\cdot \gamma_q[y]$ fixes $y$.
\end{proof}

 \begin{Lem}\label{l:strengtheningOrbEquiv}    Suppose that a function $g$ is 
  given  on left closed, right open  intervals by elements of a Fuchsian group $G$ acting on $\mathbb S^1$,  and that orbit equivalence holds.   Then  strong orbit equivalence holds for $g, G$.
\end{Lem}
\begin{proof}
  Let $x \in \mathbb S^1, T \in G$.  By orbit equivalence, there are $p,q \ge 0$ such that $g^p(x) = g^q(Tx)$.    Thus, $\gamma_p[x] x = \gamma_q[Tx]\, Tx$.  Since $g$ is piecewise defined,  we have that $\gamma_p[x] x' = \gamma_q[ Tx] y'$ for all $x', y'$ in sufficiently small  on left closed, right open intervals containing $x$ and $Tx$, respectively.  The action of $T$ is continuous and orientation preserving, hence  by possibly shrinking the interval containing $x$, we may assume that $T$ sends it to the interval containing $Tx$.  
That is, $\gamma_p[x]$ and $\gamma_q[ Tx] \,T$ agree on an interval.  Only the identity element of $G$ can fix any interval, and hence  $\gamma_p[x] = \gamma_q[ Tx] \,T$.
\end{proof}
 
 \begin{Lem}\label{l:strongHypFixHasFinOrb}    Suppose that a function $g$ is  given on intervals by elements of a Fuchsian group $G$ acting on $\mathbb S^1$, and strong orbit equivalence holds for $g, G$.   Then every hyperbolic fixed point of $G$ has finite $g$-orbit.
\end{Lem}
\begin{proof} Let $x$ be a fixed point of the hyperbolic element $T \in \Gamma$. There are  $p,q \ge 0$ such that  $\gamma_p[x] = \gamma_q[\, T x] \,T =  \gamma_q[\,  x] \,T$.   Since $T$ is not the identity, we have $p\neq q$.  Now, $g^p(x) = \gamma_p[x] x = \gamma_q[\,   x] \,Tx   = \gamma_q[\,  x] \, x = g^q(x)$.  It follows that $x$ has a finite $g$-orbit.
\end{proof}

\section{The extension property: failure and success; proof of Theorem~\ref{t:whichTrianglesWork} }  
\begin{proof}[Proof of Theorem~\ref{t:whichTrianglesWork}.]   By [\,\cite{beardon}, Theorem~10.5.1],  a convex fundamental domain $\mathcal{F}$ for a  triangle group $\Gamma$ is either a quadrilateral or a hexagon.  In the latter case, there must be an accidental cycle of length three.   As Pit \cite{pit} observed, such a cycle always causes a failure of the extension property; to see this,   let  $u_1,u_2$ and $u_3$ be the vertices of this accidental cycle and $\theta_1,\theta_2$ and $\theta_3$  their corresponding  internal angles. There exist $T_1,T_2\in\Gamma$ such that $T_1(u_1)=u_3$ and $T_2(u_2)=u_3$; hence $u_3\in \mathcal{F}\cap T_1(\mathcal{F})\cap T_2(\mathcal{F})$.  But,   $\theta_1+\theta_2+\theta_3=2\pi$ and for each $i$, $\theta_i<\pi$.   It follows that the geodesic that contains the side in the intersection of $T_1(\mathcal{F})$ and $T_2(\mathcal{F})$ meets the interior of $\mathcal{F}$.   Clearly, the extension property fails.

 If $\mathcal F$ is a quadrilateral, since the quotient by $\Gamma$ is of genus zero it must be that each of two opposite vertices  has its adjacent sides identified by some element of $\Gamma$.    That is,  we can label side pairing elements $T_1, \dots, T_4$ as above in the Bowen-Series setting.    By  Poincar\'e's fundamental polygon theorem, see say \cite{beardon}, the hyperbolic triangle group $\Gamma$ has the presentation $\langle a, b \mid a^{m_1} = b^{m_2} = (ab)^{m_3}\rangle$, where $a, b$ correspond to the elements $T_2, T_4$, respectively.   If now (at least) two of the $m_i$ are odd, then we can assume that $T:= T_2$ acts as a rotation about its fixed vertex, say $v$,  of angle $2\pi/m$ with $m$ an odd integer.   
 
The fact that $m$ is odd is exactly what causes the extension property to fail, and indeed reveals what motivates one in general to call this property that of `even corners'!    That all said, we give further details.   
 
 Let $s$ be the side paired to $T(s)$ by $T$,   and $s'$ the geodesic arc extending $s$ beyond $v$.  
Let $U = T^{\frac{m+1}{2}}$.   We find that $U$ sends $s'$ to the geodesic arc agreeing with a rotation by angle $2 \pi + \pi/m$ of $s$.  This arc meets the interior of $\mathcal F$; indeed, it lies on the angle bisector of the sides meeting at $v$.  Again, the extension property fails.

\begin{figure}[h]
\scalebox{.3}{
\includegraphics{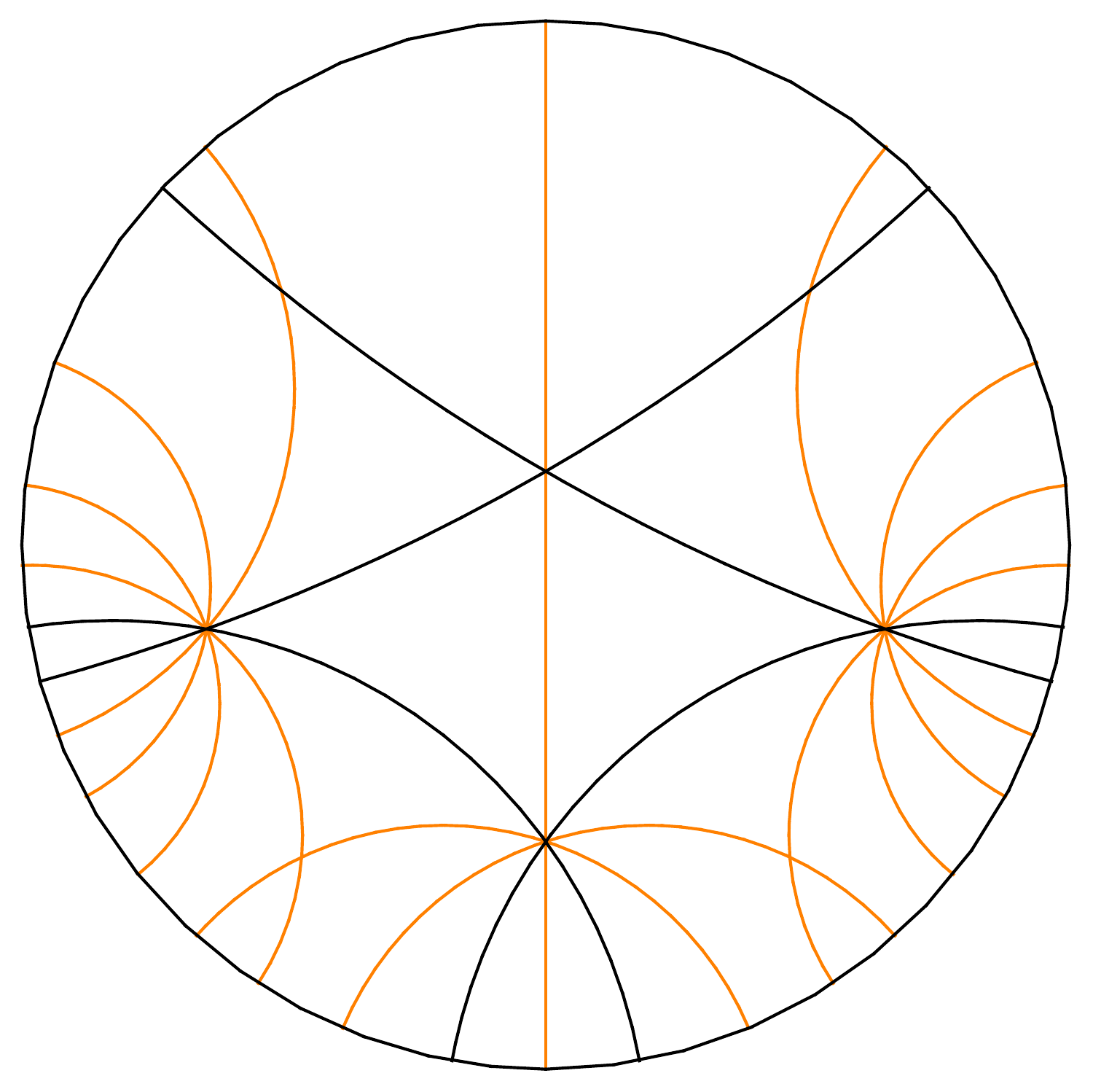} 
}
\caption{The Bowen-Series fundamental domain for  $(m_1, m_2, m_3) =   (3, 5, 6)$ fails to have the extension property.}%
\label{f:extFails}%
\end{figure}

 We now turn to the Bowen-Series fundamental domain $\mathscr F$, although we  relabel as necessary so that $m_1, m_2$ are both even.  
 We show that the sides of $\mathscr F$ do appropriately project to closed geodesics, and thus the result will follow from Lemma~\ref{l:projToGeos}. First suppose that also $m_3$ is even.   Then each of the four vertices is fixed by an elliptic element of even order,  and taking appropriate powers, each is fixed by an elliptic element of order two.    Hence, any side  lies on the axis of the hyperbolic element formed by the product of elliptic elements of order two  corresponding to its endpoints.  In particular the side itself projects to a closed geodesic (on the orbifold, the closed geodesic is given by traversing the projected path twice, once in each direction).   Thus, the extension property holds in this case. 

If $m_3 = 2k + 1$ is odd, then we claim that each  union of opposite sides of $\mathscr F$ projects to a closed geodesic.    For ease, we begin with $r_4$ as emanating from $v_4$.   Following the unit tangent vectors along this side until $v_1$,  we then apply $(T_2T_4)^k$ thus rotating the unit tangent vector by $k\cdot 2 \pi/m_3 = 2k \pi /(2k+1) = \pi -\pi/(2k+1)$.   An application of $T_4$ now sends this to the unit tangent vector based at $v_2$ and pointing in the negative direction of $r_2$.   Following this side until reaching $v_2$, we apply the appropriate power of $T_2$ so as to turn the unit tangent vector through an angle of $\pi$.   We now retrace the steps all the way back to $v_4$, where the arriving unit tangent vector has  the opposite direction from that of the initial departing vector; we thus apply the appropriate power of $T_4$ so as to turn the unit tangent vector through an angle of $\pi$.    In summary,   we find that the union of $r_4$ and $-r_2 = \ell_3$ projects to the closed geodesic that is the projection of the axis of the hyperbolic element that this the product of $[T_{4} (T_2T_4)^{k}]^{-1}\, T_{2}^{m_2/2} \,[T_{4} (T_2T_4)^{k}]$ with $T_{4}^{m_1/2}$.  By symmetry, a similar result holds for the remaining two sides, and we find that the extension property holds also in this case. 
\end{proof}

 \begin{Rmk}\label{l:otherGenusZero}  The proof above shows that the fundamental domain constructed  in \cite{bs}    in the setting of genus $g=0$ and $n\ge 3$ singularities fails to have the extension property whenever at most one of the indices is even.  
\end{Rmk}

 Naturally enough, in what follows we will restrict to the setting where the extension property holds.  In order to also ensure that eventual expansivity holds, we further insist that  $\mathcal F$ is not a degenerate quadrilateral.  (That is, there are no vertices of angle $\pi$.) 
 
\begin{Def}\label{d:WithExtensionProperty}  Let $\mathscr E \subset  \{(m_1, m_2, m_3) \in \mathbb N^3\}$ such that 
\begin{enumerate}
\item $\sum_{i=1}^3 m_{i}^{-1} < 1$;
\item $2 \mid m_1, 2 \mid m_2$;
\item $2 \in \{m_1, m_2, m_3\} \Rightarrow m_3 = 2$.
\end{enumerate}
\end{Def}

\section{Continuous deformations of the Bowen-Series map} 
\subsection{Definition of function and partition}\label{ss:defFunPart}  
 Fix a signature in $\mathscr E$ and its  Bowen-Series function $f$, defined in \eqref{e:defF}.
For each $\alpha \in \mathscr O$, 
we define 
 \begin{equation}\label{e:ThatsOurFunction}
 f_{\alpha}(x) :=\begin{cases}
T_{i-1}(x),& x\in [a_i^{n_i},\alpha);\\
f(x),& \text{otherwise},\\
\end{cases}
 \end{equation} 
where $\alpha \in \mathscr O_i$.
Note that each $f_{\alpha}$ is expansive, since  $f$ itself is expansive and   $[a_i^{n_i},\alpha)$ lies within the isometric circle of $T_{i-1}$.    We call  $\mathscr D = \mathscr D_\alpha= [a_i^{n_i},\alpha)$ the {\em differing interval}; indeed, the functions $f, f_{\alpha}$ differ exactly on $\mathscr D$.   See Figure ~\ref{f:plots344NonSurjective} for plots of one pair $f, f_{\alpha}$.

\begin{figure}[h]
\scalebox{1.2}{
\includegraphics{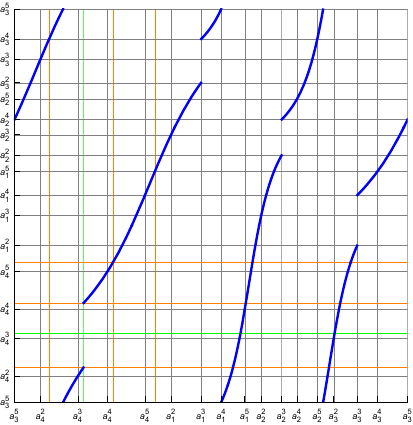} 
}
\caption{Plot representing the  function $f_{\alpha}$ for signature $(6,6,3)$  with $\alpha$ the fixed point of $T_{4}^{2}T_{2}^{2}T_{3}T_{1}T_{4}^{2}T_{1}T_{4}^{3}$.  Here, $\alpha \in M_1\cap \mathscr O_4$, and is marked by the light green lines.  Light gray lines correspond to points in $W$, light orange lines to the remaining points in $W_\alpha$.%
}%
\label{f:plot366Alp2}%
\end{figure}

   Let 
 \[ W_{\alpha} = W \cup \{ f_{\alpha}^k(\alpha)\}_{k\ge 0} \cup \{ f_{\alpha}^k(T_{i-1} \alpha)\}_{k\ge 0}.\]
 Since $W$ is $f$-invariant  and $f,f_{\alpha}$ agree on $W$ except at $a_i^{n_i}$, the following implies the $f_{\alpha}$-invariance of $W_{\alpha}$.  We have  $f_{\alpha}(a_i^{n_i})=T_{i-1}(a_i^{n_i})$,  by \eqref{wpts} this equals $a_{\rho(i)}^{n_i-1}$ and thus belongs to $W$. 
 Let $\mathcal P_\alpha$ be the partition of $\mathbb S^1$ given by the left-closed, right-open  intervals whose endpoints are the elements of $W_{\alpha}$.   See Figure~\ref{f:plot366Alp2} for an indication of one $\mathcal P_{\alpha}$.

 \begin{Rmk}\label{l:moreFun}  Note that one could take the four-dimensional family defined by choosing a parameter in each $\mathscr O_i$ and simultaneously  making changes as above to define a new function.   We intend to return to this setting in upcoming work. 
\end{Rmk}

\subsection{More on $f$-dynamics} \phantom{hi} 

  Our study of the various $f_{\alpha}$ requires various details about the dynamics of $f$ itself.

\begin{Lem}\label{l:leftRightFan}  For each $i$,  the image under $f$ of  $\mathcal A_i$ is the interval $R_1(v_{\rho(i)})\setminus R_{n_{\rho(i)-1}}(v_{\rho(i)-1})$.   The image of $L_1(v_i)$ under $f$ is given by $f(\mathcal A_i)\cup L_{n_{i+1}-1}(v_{\rho(i+1)})$.   
\end{Lem}

\begin{proof}   The restriction of $f(x)$ to $\mathcal A_i = L_1(v_i)\setminus L_{n_{i+1}}(v_{i+1})$ is given by the action of $T_i$, see \eqref{e:defAsubI}.  From \eqref{wpts}, $T_i$ sends $ L_1(v_i)$ to $R_1(v_{\rho(i)})$ and also $L_{n_{i+1}}(v_{i+1})$ to $R_{n_{i+1}}(v_{\rho(i+1)})$.   Since for all $i$ both $n_i = n_{\rho(i)}$ and $\rho(i+1) = \rho(i)-1$ hold, the first result holds.  The second follows from \eqref{e:partition}.
\end{proof}
 
The following  lemma is a  direct implication of Lemma~\ref{l:classificationOKandOrbit1} and its proof.  
\begin{Lem}\label{l:babySteps} Fix $i$.   Then for any $x\in  L_{n_i}(v_i)$,   we have
 \[
 \begin{cases} \;f^j(x)\;\;\;\;\;\; \in L_{n_i - j}(v_{\rho^j(i)}), & 0\le j \le n_i-1,\\
 \\
                        f^j(\, T_{i-1} x) \in R_{n_i - j}(v_{\rho^{1+j}(i)}), & 0\le j  \le n_i-1.                      
 \end{cases}
 \]
\end{Lem}

\begin{Lem}\label{l:yAtEndOfGiantStep} Fix $i$.   Then for any $x\in  L_{n_i}(v_i)$, setting $y= T_{i-1}(x)$ we have
\[
    f^{n_i-1}(\, y\,)  \in   \begin{cases} R_1(v_{\rho^{n_i}(i)}) \setminus R_{n_{\rho(\theta(i)+1)}}(v_{\rho(\theta(i)+1)})&\text{iff}\;f^{n_i-1}(x)\in   \mathcal A_{\theta(i)};\\
\\
R_{n_{\rho(\theta(i)+1)}}(v_{\rho(\theta(i)+1)})&\text{iff}\;f^{n_i-1}(x)\in L_{n_{\theta(i)+1}}(v_{\theta(i)+1}).
                   \end{cases}
\]
\end{Lem}
\begin{proof}   From Lemma~\ref{l:classificationOKandOrbit1}, we have  $f^{n_i-1}(y)=T_{\theta(i)}f^{n_i-1}(x)$.  The result thus follows from Lemma~\ref{l:leftRightFan}.
\end{proof}

The parity of  $m_3$ determines the period length of the sequence of indices $(i_k)_{k\ge 0}$.  Recall that for $k>0$ we have $i_k = \theta(i_{k-1}) +1$ where $\theta(i)=\rho^{n_i-1}(i)$ for each $i$.
 \begin{Lem}\label{l:periodOfIsubK}  Fix any $i_0 \in \{1,2,3,4\}$.   The sequence $(i_k)_{k\ge 0}$ is purely periodic.  Furthermore, the period length is 2 if $m_3$ is even and 4 otherwise.   Taken in order, the period alternates between even and odd values.  
 \end{Lem}

\begin{proof}
From \eqref{e:ThetaValues}, we have 
\begin{equation}\label{e:nextI}
i_{k+1} =  \begin{cases}  i_k+1&\text{ if } 2 \mid i_k;\\
                                        i_k+1&\text{ if }  2 \nmid i_k, \, 2 \nmid n_{i_k}\,;\\
                                        i_k-1&\text{ if }  2 \nmid i_k, \, 2 \mid n_{i_k}\,.
                    \end{cases}
\end{equation}      
Therefore, the sequence $(i_k)_{k\ge 0}$ is always periodic; see Figure~\ref{f:seqIsPeriodic}.  Since  $m_3 =n_1 = n_3$, we find that the period length is 2 when $2\mid m_3$ and 4 otherwise.    The final statement follows from the fact that   $i_{k+1}$ always has the opposite parity of $i_k$.
\end{proof}
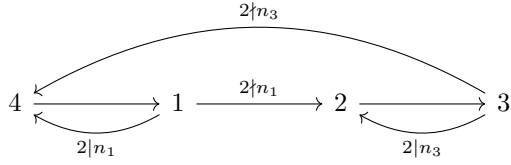
\begin{figure} 
\begin{tikzcd}[column sep=4pc,row sep=2pc]
4  \arrow{r}&1 \arrow[bend left]{l}{2\vert n_1}\arrow{r}{2\nmid n_1}&2\arrow{r}&3\arrow[bend left]{l}{2\vert n_3}\arrow[bend right]{lll}[swap]{2\nmid n_3}
\end{tikzcd}
\caption{An automaton which produces the sequence of the $i_k$, and shows that this sequence is always periodic.  Since $n_1 = n_3= m_3$, the sequence is  always purely periodic, of period length either $2$ or $4$.}
\label{f:seqIsPeriodic}%
\end{figure}

\section{Aperiodicity; the proof of Theorem~\ref{t:aperiodic}} 

 Recall, say from \cite{af},  that a function $g$ is called {\em aperiodic}  with respect to a partition of its domain of definition if there is a finite compositional power of the function that maps the closure of each partition element onto the whole domain.   Each  Bowen-Series function $f$ is shown in \cite{bs} to satisfy $\bigcup_{r=0}^\infty f^r(I)=\mathbb S^1$ for each $I$ of  $\mathcal{P}$.  Aperiodicity itself  follows from the finiteness of $\mathcal P$ and the fact that $f$ is eventually expansive and nowhere contractive.

\begin{figure}[h]
 \centering
\begin{minipage}{.5\textwidth}
  \centering
\scalebox{0.9}{
\includegraphics{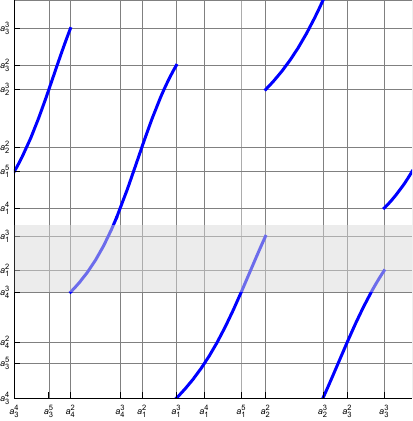}
}
\end{minipage}%
\begin{minipage}{.5\textwidth}
  \centering
  \scalebox{0.9}{
\includegraphics{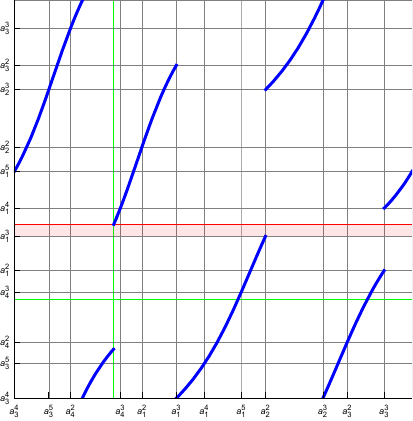}
}
\end{minipage}%
\caption{Plots of $f$ and of $f_{\alpha}$ for the signature $(4,4,3)$,  with  $\alpha$ the fixed point of $T_3 T_2 T_4 T_1 T_3 T_1T_4 T_1 T_3 T_{2}^{2} T_3 T_1T_4$.  This latter function is not aperiodic with respect to its Markov partition $\mathcal P_{\alpha}$; it is not even a surjective function! On left in gray  $f(\mathscr D)$;  on right in red,  the subinterval $[a_{1}^{3}, f(\alpha))$ of $f(\mathscr D)$ is not  in the image  of $f_{\alpha}$.%
}%
\label{f:plots344NonSurjective}%
\end{figure}

\subsection{Surjectivity of $f_{\alpha}$}  

 We first identify a condition that holds exactly when $f_{\alpha}$ is surjective.  Compare the following with Figure~\ref{f:plots344NonSurjective}.
\begin{Lem}\label{l:surjIndentified} Suppose $\alpha\in\mathscr{O}_i$.  The function $f_{\alpha}$ is surjective if and only if there is a union $J$ of elements of  $\mathcal P \setminus \{\mathscr O_i\}$ such that  $f(\mathscr D) \subseteq f(J)$. 
\end{Lem}

\begin{proof}   Since $f_{\alpha}(\mathscr D) \cap f(\mathscr D) = \emptyset$ and also $f$ is given by $T_i$ on all of $\mathscr O_i \supset \mathscr D$,  the map $f_{\alpha}$ is surjective if and only if $f(\mathscr D)$ is contained in the  image of $\mathbb S^1 \setminus \mathscr  O_i$ under $f$.    Since $\mathbb S^1 \setminus \mathscr  O_i$ admits the partition $\mathcal P \setminus \{\mathscr O_i\}$,  the condition can be relaxed to $f(\mathscr D)$ being contained in the $f$-image of some union of elements of $\mathcal P \setminus \{\mathscr O_i\}$.     
\end{proof}

We now give two further results about $f$-dynamics before showing, using the reduction of the previous lemma,  that the conditions of Theorem~\ref{t:aperiodic} are those identifying the $\alpha$ for which $f_{\alpha}$ is surjective.
\begin{Lem}\label{l:ContainementofNonoverlapIntervals}  
Fix $i$. The interval $f(\mathcal A_{\rho(i-1)})$ contains  $\mathcal A_i$ and  each of $L_k(v_{i}),  2\leq k\leq n_i-1$. When $n_i>2$, a leftmost subinterval of $f(\mathcal A_{\rho(i-1)})$  is formed by $L_{n_i-1}(v_{i})$; otherwise, $n_i=2$ and $\mathcal A_i$ forms a leftmost subinterval of $f(\mathcal A_{\rho(i-1)})$.

 Furthermore, $\mathcal A_{\rho(i-1)}$ meets none of the $L_{j}(v_{\rho(i)}), 2\le j\le n_i$.
\end{Lem}

\begin{proof}
From Definition~\ref{LRpartition}, one finds  
\[\mathcal A_i   \cup \bigcup_{k=2}^{n_{i}-1}L_k(v_{i}) \subseteq R_1(v_{i-1})\setminus R_{n_{i-2}}(v_{i-2}).\] (The case of equality occurs exactly when  $n_{i-2}=2$.) Since $\rho^2$ is the identity,  by  \eqref{e:partition} and Lemma~\ref{l:leftRightFan}  we have  the  containment  in $f(\mathcal A_{\rho(i-1)})$.   Definition~\ref{LRpartition} then also implies the statement on the leftmost subintervals.

Finally,  again by the definition of the intervals, $\mathcal A_{\rho(i-1)}= L_1(v_{\rho(i-1)})\setminus L_{n_{\rho(i-1)+1}}(v_{\rho(i-1)+1})$ fully contains no $L_j(v_k)$.  But, each $L_{j}(v_{\rho(i)}), 2\le j\le n_i$ is an element of $\mathcal P$ and $\mathcal A_{\rho(i-1)}$ is a union of elements of $\mathcal P$.  Therefore, by the Markov property of $f$, it must be that  $\mathcal A_{\rho(i-1)}$ does not intersect any of the $L_{k+1}(v_{\rho(i)}), 2 \le k \le n_{i}-1$. 
\end{proof}

\begin{Lem}\label{l:overlapIntervalContainment}
Fix $i$.  The overlap interval  $\mathscr O_i$ is contained in $f( \mathcal A_{\rho(i+2)})$  if and only if $n_{i+1}>2$.
\end{Lem}

\begin{proof}   
By Lemma~\ref{l:leftRightFan}, $f(\mathcal A_{\rho(i+2)}) = R_1(v_{i+2})\setminus R_{n_{i+1}}(v_{i+1})$. 
By Lemma~\ref{l:keyContainers}, this set fails to fully contain $\mathscr O_i$ exactly when $n_{i+1}= 2$.   
\end{proof}

We now show that the conditions of Theorem~\ref{t:aperiodic} are those identifying the $\alpha$ for which $f_{\alpha}$ is surjective.  We use an overline over the notation for a set to indicate the closure of the set. 
\begin{Lem}\label{l:DifferingIntervalContainement}
Suppose $\alpha\in\mathscr{O}_i$. There exists a union $J$ of elements of $\mathcal P\setminus \{\mathscr O_i\}$ with  $f(\mathscr D)\subset f(J)$ if and only if any of the following conditions is satisfied:
\begin{itemize}
\item[(i)] $n_i>2$,
\item[(ii)] $n_i=2$ and $n_{i+2}>2$, 
\item[(iii)] $\alpha\in \overline{M_1}$.
\end{itemize}
When any of these conditions holds, $J$  may be taken to be $\mathcal A_{i+1}$. 
\end{Lem}

\begin{proof}  

\noindent 
($\Leftarrow$) From \eqref{e:partition},  $f(\mathscr D)\subseteq f(\mathscr O_i)= L_{n_i-1}(v_{\rho(i)})$.  From Lemma~\ref{l:keyContainers}  $\mathcal A_{i+1}$ is disjoint from $\mathscr O_i$.

When $n_i>2$, Lemma~\ref{l:ContainementofNonoverlapIntervals} shows that $f(\mathcal A_{\rho(\rho(i) -1)})$ contains  $f(\mathscr O_i)$.  One easily verifies the identity $\rho(\rho(i) -1) = i+1$.   Thus, $f(\mathcal A_{i+1})  \supset f(\mathscr{D})$.

 If $n_i=2$, then Lemma~\ref{l:ContainementofNonoverlapIntervals} shows that $f(\mathcal A_{i+1})$ contains the subset $\mathcal A_{\rho(i)}$ of $f(\mathscr O_i)$.    
Since here $\theta(i) = \rho(i)$, by Lemma~\ref{l:classificationOKandOrbit1}   $f(\mathscr{D})\subsetneq \mathcal A_{\rho(i)}$  occurs exactly when there is matching  at the first giant step for $\alpha$; namely, when $\alpha\in M_1$.   Since $f$ is orientation preserving,   $f(\mathscr{D})= \mathcal A_{\rho(i)}$ holds exactly when $\alpha \in \overline{M_1}$.  
 
Suppose now that $n_i=2$ and $n_{i+2}>2$; note that this implies  $2\mid i$.   From the previous paragraph, we are reduced to the case that  $f(\mathscr{D})$ is not fully contained in $\mathcal A_{\rho(i)}$.  By definition of the $\mathcal A_j$,  we have $f(\mathscr O_i)  = \mathcal A_{\rho(i)}\cup \mathscr O_{\rho(i)+1}$.
Since $n_{\rho(i)+2} = n_{i+2}>2$, Lemma \ref{l:overlapIntervalContainment} shows that $\mathscr O_{\rho(i)+1}$ is contained in $f(\mathcal A_{\rho(\rho(i)+1+2)})=f(\mathcal A_{i+1})$.  
We have already seen that $\mathcal A_{\rho(i)}\subseteq f(\mathcal A_{i+1})$.    Therefore,
\[f(\mathcal A_{i+1}) \supset \mathcal A_{\rho(i)} \cup \mathscr O_{\rho(i)+1} \supset f(\mathscr{D}).\]

\noindent 
($\Rightarrow$) When none of the conditions hold, we have $n_i = n_{i+2}=2$ and $\alpha\notin \overline{M_1}$.
 Since $n_i=2$, from the above we now have both that $f(\mathscr D)\supset \mathcal A_{\rho(i)}$ and   $f(\mathscr D) \cap \mathscr O_{\rho(i)+1} \neq \emptyset$.   Suppose that $J$ is a union of elements of $\mathcal P \setminus \{\mathscr O_i\}$ such that $f(J) \supset f(\mathscr D)$.  Since $\mathscr O_{\rho(i)+1} \in \mathcal P$,  by  the Markov property of $f$, we have  $f(K) \supset \mathscr O_{\rho(i)+1}$ for some partition element $K$ of the union $J$.   

Recall from  \eqref{e:leftRightPartition} that for each $k$ the partition given by the $R_j(v_k)$ and $L_j(v_k)$ with $1 \le j\le n_k$ is coarser than $\mathcal P$.  Since  $\mathcal P$ is a Markov partition,  $K$ is contained in some such interval, say $K'$.    By
Lemma~\ref{l:keyContainers},   the right or left intervals containing $\mathscr O_{\rho(i)+1}$   are exactly $L_1(v_{\rho(i)}), R_1(v_{\rho(i)+3})$, and  $R_2(v_{\rho(i)+2})$.   
There are three `obvious possibilities' for $K'$:
$L_2(v_i) =  \mathscr O_i; R_2(v_{i+1})$; and, $R_3(v_{i+2} )$ were it to exist, however since $n_{i+2}=2$ it does not.   Since $L_2(v_i)$ is itself in $\mathcal P$, we find that $K$ would be equal to this interval, but this contradicts $K$ avoiding $\mathscr O_i$.  Similarly, $R_2(v_{i+1})\supset \mathscr O_i$ and furthermore, $f$ on $R_2(v_{i+1})\setminus \mathscr O_i$ avoids $\mathscr O_{\rho(i)+1}$; this remaining obvious possibility is ruled out.  

 We now rule out any $K'$ being any $L_1(v_k)$ or $R_1(v_k)$.   It is easily verified that any $R_1(v_k)$ is the union of various left intervals, hence we need only consider the possibility that  $K'= L_1(v_k)$ for some $k$.   
By Lemma~\ref{l:leftRightFan},   if $f(L_1(v_k)) \supset \mathscr O_{\rho(i)+1}$ then $f(\mathcal A_k)\supset \mathscr O_{\rho(i)+1}$.    But by this same lemma combined with Lemma~\ref{l:overlapIntervalContainment},  $f(\mathcal A_k)\supset \mathscr O_{\rho(i)+1}$  if and only if both is $k = \rho(\rho(i)+2)$ and $n_k>2$.   However,  $\rho(\rho(i)+2) = i+2$ and thus here $n_k=2$. That is, the containment does not occur.    We have shown that there is no possible $K'\supseteq K$.    
 
Therefore, there is no possible union of elements of $\mathcal P$ which both avoids $\mathscr O_i$ and has image containing $f(\mathscr D)$.
\end{proof}

\subsection{From surjectivity to aperiodicity; completion of the Proof of Theorem~\ref{t:aperiodic}}
\begin{proof}[Proof of Theorem~\ref{t:aperiodic}]    A non-surjective function cannot be aperiodic, thus  when the conditions of  Theorem~\ref{t:aperiodic}  are not fulfilled, Lemma~\ref{l:DifferingIntervalContainement} shows that $f_{\alpha}$  is certainly not aperiodic with respect to $\mathcal P_{\alpha}$.

Now assume that the conditions are  fulfilled,  and thus $f_{\alpha}(\mathcal A_{i+1}) = f(\mathcal A_{i+1}) \supseteq f(\mathscr{D})$. 
\medskip

\noindent
{\bf Main Step: $f_\alpha$-orbit of $\mathscr D$  meets  $\mathcal A_{i+1}$.} 
By definition of $f_{\alpha}$, we have  $f_\alpha(\mathscr{D})\subset R_{n_i}(v_{\rho(i)})$, as a leftmost subinterval. 
As usual, let $i=i_0$.  
 By Lemma~\ref{l:babySteps},  we have  $f^{n_i-2}$ sends $R_{n_i}(v_{\rho(i)})$ to $R_2(v_{\theta(i)})=R_2(v_{i_1-1})$.
 By Lemma~\ref{l:keyContainers}, $\mathscr O_{i_1-2}$ is  a leftmost subinterval of $R_2(v_{i_1-1})$. Since $i_1\neq i+2$ by \eqref{e:nextI}, $\mathscr O_{i_1-2}\neq \mathscr{O}_i$.    From Lemma~\ref{l:babySteps} and \eqref{e:landingPattern}, the $f$-orbit of a leftmost subinterval of $\mathscr O_{i_1-2}$ reaches, and contains, $\mathcal A_{\theta(i_1-2)}$  before encountering any other overlap intervals.   Direct calculation shows that $\theta(i_1-2) = i_0+1$.   Therefore,  the $f_{\alpha}$-orbit of  $\mathscr D$ reaches   $\mathcal A_{i+1}$.

\medskip
\noindent
{\bf Step 2: $f_\alpha$-orbit of $\mathscr{D}$ contains $f(\mathscr{D})$.}
We have shown that the $f$-orbit of a leftmost subinterval of $f_{\alpha}(\mathscr{D})$ reaches   a leftmost subinterval of $\mathcal A_{i+1}$.  

  We now describe  a sequence of $\mathcal A_{j_k}, k\ge 0$, beginning at any $\mathcal A_{j_0}$, where for each $k$,  there is a leftmost subinterval of $\mathcal A_{j_{k}}$ having $f$-orbit reaching, and containing, $\mathcal A_{j_{k+1}}$ before encountering $\mathscr O$.    Indeed, by Lemma~\ref{l:ContainementofNonoverlapIntervals},  for general $j$, either  $f(\mathcal A_j) \supset \mathcal A_{\rho(j)+1}$ or else $f(\mathcal A_j)$ has a leftmost subinterval given by $L_{n_{\rho(j)+1}-1}(v_{\rho(j)+1})$.  In this latter case, by Lemma~\ref{l:classificationOKandOrbit1}, applying $f^{n_{\rho(j)+1}-2}$ sends a leftmost subinterval surjectively to $\mathcal A_{\ell}$, where $\ell  = \rho^{n_{\rho(j)+1}-2}(\rho(j)+1)$; note that $\mathscr O$ is not encountered during the intermediate steps of applying $f$.   In fact,  this single formula for $\ell$  applies in both of these cases: we pass from a leftmost subinterval of $\mathcal A_j$ to full containment of $\mathcal A_{\ell}$.    That is,  inductively defining $j_{k+1} = \rho^{n_{\rho(j_k)+1}-2}(\rho(j_k)+1)$ results in our desired sequence.   
  
One easily verifies that reversing the arrows in the automaton for $(i_k)_{k\ge 0}$, given in Figure~\ref{f:seqIsPeriodic}, results in an automaton giving the sequence $(j_k)_{k\ge 0}$.    We now desire to show that if  $j_0 = \theta(i_1-2)$, then there is some $k$ such that $j_k = i+1$.   If $2\nmid m_3$, then the sequence takes on all values in $\{1,2,3,4\}$.   Otherwise,  the $i_k$ and the $j_k$ sequences are both of period length two. 

 We can now follow the $f_\alpha$-orbit of $\mathscr{D}$ into the periodic sequence of the $\mathcal A_{j_k}$ beginning with $j_0 = i+1$ and from the expansive aspect of $f$ deduce that this orbit eventually contains all of $\mathcal A_{i+1}$.   Since $\mathscr O$ is avoided while doing this,  we deduce that the $f_\alpha$-orbit of $\mathscr{D}$ contains $f(\mathscr{D})$.

\medskip
\noindent
{\bf Step 3: Aperiodicity holds.}   (This is only in this step that we make use of the hypothesis that $f_{\alpha}$ is Markov with respect to $\mathcal P_{\alpha}$.) We aim to show that the $f_{\alpha}$-orbit of any $I'\in \mathcal P_{\alpha}$ includes all of $\mathbb S^1$.   Let $I' \in \mathcal P_{\alpha}$.  By hypothesis, the partition endpoint set $W_{\alpha}$ is finite.  Thus, the left endpoint of $I'$ has a finite $f_{\alpha}$-orbit, and hence contains a periodic point. Let $J' \in \mathcal P_{\alpha}$ have this periodic point as its left endpoint.  Then by the expansive aspect of $f_{\alpha}$,  applying $f_{\alpha}$ to compositional powers that are multiples of the minimal period length of the endpoint, the image of $J'$ must eventually contain some $I \in \mathcal P$.   Since $f$ is known to be aperiodic with respect to $\mathcal P$, all of $\mathbb S^1$ is contained in a finite union of powers of $f$ applied to $I$.   
\bigskip

If all but possibly the final power of $f$ so applied   are such that their image of $I$ avoids $\mathscr D$, then  
$f_{\alpha}$ would agree with $f$ throughout, and we find that all of $\mathbb S^1$ is contained in a finite union of powers of $f_{\alpha}$ applied to $I$.     Otherwise, each time that a piece of the $f$-orbit of $I$ intersects $\mathscr D\subset \mathscr O_i$, by the Markov property of $f$, that piece of the $f$-orbit meets all of $\mathscr D$.   By Step 2, we can thus follow the $f_{\alpha}$-orbit of $\mathscr D$  until we agree with $f(\mathscr D)$ and thereafter continue following the $f$-orbit.    Therefore, the $f_{\alpha}$-orbit of $I$, and hence that of $I'$, contains all of $\mathbb S^1$.  
\end{proof}

\section{Markov values; the proof of Theorem~\ref{t:markoffIsHyper} }
 \subsection{Markov values are hyperbolic fixed points}     
 
 Arguing as in \cite{bs}, the function $f_{\alpha}$ is Markov with respect to  $\mathcal P_{\alpha}$ if and only if $W_{\alpha}$ is a finite set.   The finiteness of $W_{\alpha}$ implies that $\alpha$ is a hyperbolic fixed point.      
 \begin{Lem}\label{l:finiteImpliesHyperbolic}   If $f_{\alpha}$ is Markov with respect to $\mathcal P_{\alpha}$ then $\alpha$ is a hyperbolic fixed point. 
 \end{Lem}
\begin{proof}    For $f_{\alpha}$ to be Markov, $W_{\alpha}$ must be finite.  In particular, the $f_{\alpha}$-orbit of $\alpha$ must be finite.  By Lemma~\ref{l:strongExpanGivesPerIffHyper} we have that $\alpha$ is a fixed point of $\Gamma$.    Since $\Gamma$ has no parabolic elements, this must be a hyperbolic fixed point.  
\end{proof}

 \subsection{Orbit equivalence and hyperbolic fixed points}
Under orbit equivalence, the converse of the previous result is easily shown to hold. 
 \begin{Lem}\label{l:orbitEquivGivesHyperMark}   Suppose that  $f_{\alpha}$ is orbit equivalent to the action of $\Gamma$ on $\mathbb S^1$.  
 If $\alpha$ is a hyperbolic fixed point of $\Gamma$, then   $f_{\alpha}$ is Markov with respect to $\mathcal P_{\alpha}$.
 \end{Lem}

\begin{proof}  The set $W_{\alpha}$ is finite if and only if each of $ \alpha$ and $ T_{i-1} \,\alpha$ have finite $f_{\alpha}$-orbits.  Since $\alpha$ is a fixed point, then so is its image $T_{i-1} \,\alpha$.     Lemmas ~\ref{l:strengtheningOrbEquiv} and ~\ref{l:strongHypFixHasFinOrb} now show that both $\alpha$ and $T_{i-1} \,\alpha$   have finite $f_{\alpha}$-orbit.
\end{proof}

 We argue more directly, since our maps in general do {\em not} enjoy the property of being orbit equivalent to the action of $\Gamma$, see the second named author's PhD dissertation, \cite{aykDiss}.
 
 \subsection{Hyperbolic fixed points give Markov maps, general setting}    To prove the remaining direction of Theorem~\ref{t:markoffIsHyper}, we will in fact prove the following. 
 
 \begin{Thm}\label{t:infiniteCollisions}  Fix $\alpha \in \mathscr O$.  Suppose that $x\in \mathbb S^1$ has infinite $f_{\alpha}$-orbit.  Then there are infinitely many values of $j$ such that the $f$-orbit of $x$ contains 
 either (1) $f^{j}_{\alpha}(x)$,  (2) $f\circ f^{j}_{\alpha}(x)$, or (3) $f^2\circ f^{j}_{\alpha}(x)$. 
 \end{Thm}

\noindent
\begin{proof}[Theorem~\ref{t:infiniteCollisions} implies completion of proof of Theorem~\ref{t:markoffIsHyper}]
 By the Pigeonhole Principle, we can then assume one of the three cases occurs for infinitely many values of $j$.  Since  $f$ is a  finite-to-one function, if  the $f$-orbit of $x$ is finite, then there are only finitely many preimages under $f$ or $f^2$ of this finite set.  It  then follows that whenever the $f$-orbit of $x$ is finite, there are some $j \neq k$ such that $f_{\alpha}^{j}(x)  = f_{\alpha}^{k}(x)$.  The finiteness of 
the $f_{\alpha}$-orbit of $x$ then follows.   Due to the strong orbit equivalence of $f$ and $\Gamma$, every hyperbolic fixed point of $\Gamma$ has finite $f$-orbit.   Given that $\alpha$ is a hyperbolic fixed point of $\Gamma$ then so is $T_{i-1}(\alpha)$, it follows that both of these values also have finite $f_{\alpha}$ orbits.    This in turn implies that $f_{\alpha}$ is Markov, and thus the proof of Theorem~\ref{t:markoffIsHyper} will indeed be complete. 
 \end{proof} 
 
 Of course if the $f_{\alpha}$-orbit of $x$ never meets the differing interval $\mathscr D$, then this orbit {\em is} the $f$-orbit of $x$.  Theorem~\ref{t:infiniteCollisions}  thus trivially holds for such $x$.   Each time that the $f$-orbit of $x$  enters $\mathscr D \subset \mathscr O_i$, it does so in some $M_{\ell}$.  Hence our proof will focus on $f$-orbits of points in $M_{\ell}$.

\subsubsection{Proof of Theorem~\ref{t:infiniteCollisions} when $m_3$ is even}   In this subsection we show that Theorem~\ref{t:infiniteCollisions}   holds in the case of `easy case' of even $m_3$.

\begin{Lem}\label{l:giantStepsNotThereEvenM3}  Suppose that $m_3$ is even.   
Fix $i= i_0$ and let $x = x_0 \in M_{\ell}$.  Then for $0\le k<\ell$ we have $y_k \notin \mathscr O_{i_0}$. 
\end{Lem}

\begin{proof} The case of $k = 0$ is clear, since $\mathscr O_{i_0}$ has empty intersection with its image under $T_{i_0-1}$. 

For $0<k<\ell$,  arguing as in the proof of Lemma~\ref{l:allMatching}, allows Lemma~\ref{l:yAtEndOfGiantStep} to show that  $y_k \in R_{n_{\rho(i_k)}}(v_{\rho(i_k)})$.  In view of Lemma~\ref{l:keyContainers}, since every $n_j>1$ always holds,    only if both $n_{\rho(i_k)} = 2$ and $\rho(i_k) = i_0+1$ could $y_k$ possibly lie in $\mathscr O_{i_0}$.   Since  $\rho = (13)$ as an element of the permutation group $S_4$, we have $\rho(i_k) = i_0+1$ if and only if  $(i_0, i_k) \in \{(1,2), (2,1), (3,4), (4,3)\}$.   Since $m_3$ is even,  $(i_k)_{k\ge 0}$  has period length two with period a permutation of one of $(1,4)$ and $(2,3)$.    It follows that $y_k \notin \mathscr O_{i_0}$.    
\end{proof} 

 \begin{Lem}\label{l:babyStepsWhenM3IsEven}   Suppose that $m_3$ is even.        Fix $i= i_0$ and let $x = x_0 \in M_{\ell}$ for some $\ell \in \mathbb N$.  Then for all $0\le k <\ell$ and $0< j < n_{i_k}-1$, we have   $f^j(y_k) \notin \mathscr O_{i_0}$. 
 \end{Lem}
\begin{proof}     First notice that if $n_{i_k}= 2$, then the condition $0< j < n_{i_k}-1$ is never fulfilled.
We assume that $k$ is such that the voidness of the condition on $j$ is avoided.

In view of  Lemmas~\ref{l:babySteps}  and ~\ref{l:keyContainers}, the only non-trivial case is $j = n_{i_k}-2$, where we then know that $f^{n_{i_k}-2}(y_k) \in  R_2(v_{\theta(i_k)})$, an interval which meets    $\mathscr O$ exactly in  $\mathscr O_{\theta(i_k)-1}$.   Thus, if we can show that $i_0 \neq \theta(i_k)-1$ we are done.  From \eqref{e:indexSequence},  $\theta(i_k)-1 = i_{k+1}-2$.

 By Lemma~\ref{l:periodOfIsubK}, the sequence $(i_k)_{k\ge 0}$ is purely periodic of length two, with one odd value and one even.     Now, whatever the value of $i_0 \in \{1,2,3,4\}$ we cannot have that $ i_{k+1}-2 = i_0$.   The result thus holds. 
\end{proof}

We now show  that Theorem~\ref{t:infiniteCollisions}   holds in this case of even $m_3$. 
 \begin{Lem}\label{l:freqMeetsWhenM3IsEven}   Suppose that $m_3$ is even and  that $x\in \mathbb S^1$ has infinite $f_{\alpha}$-orbit.  Then there are infinitely many values of $j$ such that the $f$-orbit of $x$ contains 
 $f^{j}_{\alpha}(x)$. 
 \end{Lem}
\begin{proof} As mentioned above, if the $f_{\alpha}$-orbit of $x$ never enters $\mathscr D$, then this orbit and the $f$-orbit agree.  Thus in this case, the result clearly holds.   Otherwise,  by temporarily renaming some element of this $f_{\alpha}$-obit  as $x$ as necessary, we may  suppose that there is some $\ell$ such that $x \in M_{\ell} \cap \mathscr D$.   Set $y = f_{\alpha}(x) = T_{i_0-1}(x)$.  By the previous two lemmas, we have that the $f$-orbit of $y$ avoids $\mathscr O_{i_0}$ and hence certainly $\mathscr D$, until at least $f^{r_{\ell}}(y)$.  In particular,  $f^{r_{\ell}}(y) = f_{\alpha}^{r_{\ell}}(y)$. Now, by definition of $M_{\ell}$ we have that $f^{r_{\ell}+1}(x) = f^{r_{\ell}}(y)$.   Since  $f^{r_{\ell}+1}(x) = f_{\alpha}^{r_{\ell}}(y) = f_{\alpha}^{r_{\ell}+1}(x)$, we have a matching associated to this entry into $\mathscr D$.

If there  are only finitely many  entries into $\mathscr D$ of the $f_{\alpha}$-orbit of $x$, then there is a last matching as above.  We  take the values of $j$ giving the  infinite tail that begins with this matching of the orbits.  Finally, if  the orbit returns infinitely often to $\mathscr D$ then we take those $j$ giving the corresponding matching values. 
\end{proof}

 \subsubsection{Odd $m_3$.}   In this subsection we show that Theorem~\ref{t:infiniteCollisions}   holds in the remaining case of odd $m_3$.

 \begin{Lem}\label{l:nAllBigFindMatches} 
Suppose that $m_3$ is odd. Fix $i=i_0$ and let $x=x_0\in M_\ell$. If 
\begin{itemize}
\item[(i)] $2\nmid i$ or
\item[(ii)] $2\mid i$ and at least one of $m_1$, $m_2$ is not equal to 4,
\end{itemize}
 then for all $0\leq k<\ell-1$ and $0<j<n_{i_k}-1$, we have $f^j(y_k)\notin\mathscr{O}_{i_0}$.  
 
  Moreover, if $k= \ell-1$ then for $0<j<n_{i_k}-1$ the containment of $f^j(y_{\ell-1})$ in $\mathscr{O}_{i_0}$ is only possible if $\ell \equiv 2\,(\text{mod } 4)$ and $j=n_{i_k}-2$.
\end{Lem}
\begin{proof} The result is clear for $k=0$.  Let $0<k\le \ell-1$.
As in the proof of Lemma~\ref{l:babyStepsWhenM3IsEven},  if $i_0\neq i_{k+1}-2$ then the result holds.   Since $m_3$ is odd, the period length of $(i_k)_{k\geq 0}$ is four. Furthermore, $i_0=i_{k+1}-2$ (with our usual modulus 4 convention) exactly when $k+1\equiv 2 \,(\text{mod } 4)$. Therefore if $k\not\equiv 1\,(\text{mod } 4)$, we have $f^j(y_k)\notin\mathscr{O}_{i_0}$, where $0<j<n_{i_k}-1$.

Suppose  $k\equiv 1\,(\text{mod } 4)$, thus $i_{k+1}=i+2$. Moreover, suppose that $f^j(y_k)\in\mathscr{O}_{i_0}\subset L_{n_i}(v_i)$. Again as in the proof of Lemma~\ref{l:babyStepsWhenM3IsEven}, we  must then have $j=n_{i_k}-2$;  it follows from Lemma~\ref{l:babySteps}  that $y_{k+1}\in L_{n_i-1}(v_{\rho(i)})$. 
On the other hand, Lemma~\ref{l:yAtEndOfGiantStep} gives two cases as to where $y_{k+1}$ is to be found. Upon replacing $x$ there by $x_k$, these cases correspond to the end of a giant step. But, by Lemma~\ref{l:classificationOKandOrbit1} exactly the first case implies that $f(x_{k+1})=y_{k+1}$. By the minimality of $\ell$, we must then have $k=\ell-1$.   In particular, the final statement of our  result holds.

 Now with $k\equiv 1\,(\text{mod } 4)$ and $k<\ell-1$, we are reduced to  considering the possibility of $y_{k+1}\in R_{n_{\rho(i_{k+1})}}(v_{\rho(i_{k+1})})=R_{n_{\rho(i+2)}}(v_{\rho(i+2)})$.      We now argue that it is impossible for $y_{k+1}$ to be simultaneously in both $L_{n_i-1}(v_{\rho(i)})$ and $R_{n_{\rho(i+2)}}(v_{\rho(i+2)})$, compare with Table~\ref{table:Yk+1WhenkIs1mod4}.    If $2\mid i$, then by hypothesis at least one of $m_1$ and $m_2$ is not 4, and thus we have that at least one of $n_2$  and $n_4$ is greater than 2.    If $2\nmid i$ then recall that  $m_3$  being odd implies $n_1,n_3>2$.  With this, applying Definition 2.5 then shows that  in each case $L_{n_i-1}(v_{\rho(i)})\cap R_{n_{\rho(i+2)}}(v_{\rho(i+2)})=\emptyset$. That is, $f^j(y_k)\notin \mathscr{O}_{i_0}$ when $0\leq k<\ell-1$.
 \end{proof}

\begin{table}[h!]
   \begin{tabular}{c||c|c}
     \phantom{$2_{2_{2_{2_{\rho}}}}$}    & $L_{n_i-1}(v_{\rho(i)})$ & $R_{n_{\rho(i+2)}}(v_{\rho(i+2)})$ \\     
      \hline      \hline
   \phantom{$\bigg($}     $2\mid i$ & $L_{n_i-1}(v_i)$ & $R_{n_{i+2}}(v_{i+2})$\\
   \phantom{$\big($}    $2\nmid i$ & $L_{n_i-1}(v_{\rho(i)})$ & $R_{n_i}(v_i)$\\   
    \end{tabular}
    \vspace{1em}
    \caption{Candidates for the location of $y_{k+1}$ if $k\equiv 1\,(\text{mod } 4)$ in the proof of Lemma~\ref{l:nAllBigFindMatches}}
  \label{table:Yk+1WhenkIs1mod4}
\end{table} 

\begin{Lem}\label{l:oddM3EvenILowM1M2}
Suppose that $m_3$ is odd  and $m_1=m_2=4$. Suppose that $i=i_0$ is even and that $x=x_0\in M_\ell$. Then for all $0\leq k<\ell-2$ and $0<j<n_{i_k}-1$, we have $f^j(y_k)\notin\mathscr{O}_{i_0}$.

 Moreover, if $\ell-2 \le k \le  \ell-1$ then for $0<j<n_{i_k}-1$ the containment of $f^j(y_k)$ in $\mathscr{O}_{i_0}$ is only possible if $k \equiv 1\,(\text{mod } 4)$ and $j=n_{i_k}-2$.
\end{Lem}
\begin{proof}
Firstly, we note that since $m_1=m_2=4$, we have $n_2=n_4=2$. By the proof above, it suffices to check the case where both $k\equiv 1\,(\text{mod } 4)$ and $i_{k+1}=i+2$, where we again  find that  only if $y_{k+1}\in L_{n_i-1}(v_{\rho(i)}) \cap R_{n_{\rho(i+2)}}(v_{\rho(i+2)})$ can the result fail to hold.  By our hypotheses, we have $L_{n_i-1}(v_{\rho(i)})=L_1(v_i)$ and $R_{n_{\rho(i+2)}}(v_{\rho(i+2)})=R_2(v_{i+2})$.  By Definition 2.5, $L_1(v_i)\cap R_2(v_{i+2})=L_{n_{i+1}}(v_{i+1})$.   

Now, if $y_{k+1}\in L_1(v_i)$,  then from Lemma~\ref{l:yAtEndOfGiantStep} we have the following two cases:
\[\begin{cases}
x_{k+1}\in L_1(v_{i+1})\setminus L_2(v_{i+2})&\text{ and } y_{k+1}\in L_1(v_i)\setminus L_{n_{i+1}}(v_{i+1}),\\
x_{k+1}\in  L_2(v_{i+2}) &\text{ and }y_{k+1}\in  L_{n_{i+1}}(v_{i+1}).
\end{cases}
\]
In the first of these cases, matching occurs in that $f(x_{k+1}) = y_{k+1}$ and hence $k = \ell-1$. 
In the second case, we find that 
\[y_{k+2}=f^{n_{i_{k+1}}-1}(y_{k+1})=f^{n_{i+2}-1}(y_{k+1})=f(y_{k+1})\in L_{n_{i-1}-1}(v_{i-1})\subset R_1(v_{i+2})\setminus R_{n_{i+1}}(v_{i+1}).\]
Therefore, $y_{k+2}$ is in the first case in Lemma~\ref{l:yAtEndOfGiantStep}, and Lemma~\ref{l:classificationOKandOrbit1} then gives that  $f(x_{k+2})=y_{k+2}$. That is $k=\ell-2$.
\end{proof}

\begin{Lem}\label{l:giantStepsNotThereOddM3}  Suppose that $m_3$ is odd.    Fix $i= i_0$ and let $x = x_0 \in M_{\ell}$.    Then for all $0\le k <\ell-1$,  $y_k \notin \mathscr O_{i_0}$. 
Moreover,  $y_{\ell-1} \in \mathscr O_{i_0}$ implies all of the following:  $2\nmid i$; $n_{i+1} =2$;  and, $\ell \equiv 2 \pmod 4$.
\end{Lem}
\begin{proof} Since $\mathscr O_{i_0}$ has empty intersection with its image under $T_{i_0-1}$, we certainly have that $y_0 \notin \mathscr O_{i_0}$.

For $0<k<\ell$,  as in the proof of Lemma~\ref{l:giantStepsNotThereEvenM3},  only if both $n_{\rho(i_k)} = 2$ and $\rho(i_k) = i_0+1$ could $y_k$ possibly lie in $\mathscr O_{i_0}$.   From this, we must have $n_{i_0+1} = 2$.  
Since $m_3$ is odd, both $n_1$ and $n_3$ are greater than two.     Therefore,   $y_k \notin \mathscr O_{i_0}$ when $2\mid i_0$.   
 
Also  as in the proof of Lemma~\ref{l:giantStepsNotThereEvenM3},  $\rho(i_k) = i_0+1$ if and only if  $(i_0, i_k) \in \{(1,2), (2,1), (3,4), (4,3)\}$.   Since $m_3$ is odd,  $(i_k)_{k\ge 0}$  has period length four with period a permutation of $(1,2,3,4)$.    Thus, $y_k \in \mathscr O_{i_0}$ implies  $k\equiv 1 \pmod 4$.  

Now,  $y_k \in \mathscr O_i, 2 \nmid i, n_{i+1} =2$ and  $k\equiv 1 \pmod 4$ imply both $y_k \in R_2(v_{i+1})$ and $n_{i_k} = 2$. Thus,
\[y_{k+1} = f^{n_{i_k}-1}(y_k) = f(y_k) \in R_1(v_{i+1}).\]
Also, if $y_k \in \mathscr O_i$ then $y_k \in L_{n_i}(v_i)$ and $y_{k+1} \in L_{n_i-1}(v_{\rho(i)})\subset R_1(v_{i+1})\setminus R_{n_i}(v_i)$.   That is, $y_{k+1}$ is in the first case of Lemma~\ref{l:yAtEndOfGiantStep} and hence matching occurs.   Therefore,  $k+1 = \ell$.
\end{proof}

By arguing as at the end of the proof of Lemma~\ref{l:freqMeetsWhenM3IsEven}, the following shows that Theorem~\ref{t:infiniteCollisions} holds when $m_3$ is odd. 
\begin{Lem}\label{l:mOddMatching}  Suppose that $m_3$ is odd.    Fix $i= i_0$ and let $x = x_0 \in M_{\ell}\cap \mathscr D   \subset \mathscr O_{i_0}$.  
Then 
\[ f^{r_{\ell}+1}(x) = \begin{cases} f_{\alpha}^{r_{\ell}+1}(x)&\text{or}\\
\\
                                                       f\circ f_{\alpha}^{r_{\ell}}(x)&\text{or}\\
                                                       \\
                                                       f^2\circ f_{\alpha}^{r_{\ell}-1}(x).
                              \end{cases}                         
\]            
  Furthermore,  in the last two cases, the corresponding $f_{\alpha}^{r_{\ell}+a}(x)   \in \mathscr D$.
\end{Lem}
\begin{proof} Since $x \in M_{\ell}$ we have 
\begin{equation}\label{e:theMatch}
f^{r_{\ell}+1}(x) = f^{r_{\ell}}(y).
\end{equation}

 We now consider various cases.  Recall from Subsection~\ref{ss:matching} that we informally refer to each  $f^j(x_k)\mapsto f^{j+1}(x_k)$ with $j \le n_{i_k}-1$ as a baby step.

\smallskip
\noindent
{\bf Case of odd $i$.} \, By Lemma~\ref{l:nAllBigFindMatches},  $f^j(y_k)\notin\mathscr{O}_{i_0}$ for all 
$0\leq k<\ell-1$ and $0<j<n_{i_k}-1$.    If $n_{i+1} \neq 2$ then Lemma~\ref{l:giantStepsNotThereOddM3} implies that 
$y_k\notin\mathscr{O}_{i_0}$ for all  $0\leq k<\ell$. Arguing as in the proof of Lemma~\ref{l:nAllBigFindMatches}, the only remaining possibility for an `early' entrance to $\mathscr D$  is at  the final  baby  step before matching.  That is,    $f^s(y)  \notin\mathscr{D}$ for all  
$0\le s < r_{\ell-1}+n_{r_{\ell-1}}-2= r_{\ell}-1$.   Hence,  $f^{r_{\ell}-1}(y)= f_{\alpha}^{r_{\ell}-1}(y)= f_{\alpha}^{r_{\ell}}(x)$. 
Therefore,  \eqref{e:theMatch} gives  
\begin{equation}\label{e:shorterMatch}
 f^{r_{\ell}+1}(x) =  f\circ f_{\alpha}^{r_{\ell}}(x).
 \end{equation} 
Now, if $f^{r_{\ell}-1}(y) \notin \mathscr D$, then  $f$ and $f_{\alpha}$ agree on this value, and  one finds $f^{r_{\ell}+1}(x) = f_{\alpha}^{r_{\ell}+1}(x)$.     

In the subcase of $n_{i+1} = 2$,   the only possible early entry into $\mathscr D$ is of $y_{\ell-1}$ if also $\ell \equiv 2 \mod 4$. 
Arguing as above, we find  
\[f^{r_{\ell-1}}(y)= f_{\alpha}^{r_{\ell-1}}(y)= f_{\alpha}^{r_{\ell-1}+1}(x).\]   Since $r_{k+1} = r_{k} + n_{i_k}-1$,  here we find $f^{r_{\ell}}(y) = f^{r_{\ell-1}+1}(y) = f\circ f_{\alpha}^{r_{\ell-1}+1}(x) = f\circ f_{\alpha}^{r_{\ell}}(x)$.   
Of course, if there is no early entry in this subcase, then we argue as above.   Note, throughout this case, \eqref{e:shorterMatch} always holds.
 
\smallskip
\noindent
{\bf Case of Even $i$.} \, Other than the subcase of $m_1=m_2=4$ one again finds the same two possibilities for  $f^{r_{\ell}+1}(x)$.  Suppose now that $m_1=m_2=4$ and thus $n_2=n_4 = 2$.  Lemmas~\ref{l:oddM3EvenILowM1M2}  and ~\ref{l:giantStepsNotThereOddM3} combine to show that entry of the $f$-orbit of $y$ into $\mathscr D$ before matching can occur at most once; it is possible only for  $f^{n_{i_k}-2}(y_k)$ with $k \in \{\ell-2, \ell-1\}$ and $k\equiv 1\,(\text{mod } 4)$.

The case of $k=\ell-1$ is argued as above, except that here only the form as given in \eqref{e:shorterMatch} is possible.  Consider now the case of  $k=\ell-2$,  with of course $\ell>2$.  Here,  $ f^{r_{\ell-1}-1}(y)=  f_{\alpha}^{r_{\ell-1}-1}(y)$ holds.   Note that also  $n_i = n_{i_{\ell-1}} = 2$,  and hence $r_{\ell} = r_{\ell-1}+n_{i_{\ell-1}}-1$ gives 
\begin{equation}\label{e:matchWithTwoInExceptCase}
f^{r_{\ell}}(y)= f^{n_{i_{\ell-1}}}\circ f^{r_{\ell-1}-1}(y)= f^2\circ f^{r_{\ell}-2}(y) = f^2\circ f_{\alpha}^{r_{\ell}-2}(y).
\end{equation}   Again,  \eqref{e:theMatch} applies to give the result. 
\end{proof}

We now show  that Theorem~\ref{t:infiniteCollisions}   holds in this case of odd $m_3$. 
 \begin{Lem}\label{l:freqMeetsWhenM3IsEven}   Suppose that $m_3$ is odd and  that $x\in \mathbb S^1$ has infinite $f_{\alpha}$-orbit.  Then there are infinitely many values of $j$ such that the $f$-orbit of $x$ contains at least one of 
 $f^{j}_{\alpha}(x), f\circ f^{j}_{\alpha}(x)$, $f^2\circ f^{j}_{\alpha}(x)$. 
 \end{Lem}
\begin{proof} As mentioned above, if the $f_{\alpha}$-orbit of $x$ never enters $\mathscr D$, then this orbit and the $f$-orbit agree.  Thus in this case, the result clearly holds.   Otherwise,  by temporarily renaming some element of this $f_{\alpha}$-orbit  as $x$ as necessary, we may  suppose that there is some $\ell$ such that $x \in M_{\ell} \cap \mathscr D$.   By Lemma~\ref{l:mOddMatching}, we have that  $f^{r_{\ell}+1}(x) \in \{f_{\alpha}^{r_{\ell}+1}(x),   f\circ f_{\alpha}^{r_{\ell}}(x), f^2\circ f_{\alpha}^{r_{\ell}-1}(x)\}$.   Hence, if there are infinitely many entries of the $f_{\alpha}$-orbit into $\mathscr D$ then we are done.

If there  are only finitely many  entries into $\mathscr D$ of the $f_{\alpha}$-orbit of $x$, then  we replace $x$ by its final visit.   Since the last two cases of Lemma~\ref{l:mOddMatching} result in a further visit to $\mathscr D$,  here it can only be the case that $f^{r_{\ell}+1}(x)= f_{\alpha}^{r_{\ell}+1}(x)$.  Since there is no further visit to $\mathscr D$, the infinite tails of these orbits agree.
\end{proof}

\end{document}